\documentclass[letterpaper,10pt,oneside,onecolumn,reqno]{amsart}

\usepackage{amsmath, amssymb}
\usepackage[left=1in,top=1in,right=1in,bottom=1in]{geometry}
\usepackage{setspace}
\usepackage{hyperref}
\usepackage{enumerate}
\usepackage{wrapfig}
\usepackage{graphicx}
\usepackage[all]{xy}
\usepackage{color}
\usepackage[usenames,dvipsnames]{xcolor}
\usepackage{textgreek}
\usepackage{units}

\onehalfspace
\newcommand{\B}{\mathcal B}
\newcommand{\C}{\mathbb C}

\newcommand{\D}{\mathcal D}
\newcommand{\E}{\mathbb E}
\newcommand{\EE}{\mathcal E}

\newcommand{\I}{\operatorname{I}}

\newcommand{\N}{\mathbb N}
\renewcommand{\P}{\mathbb P}
\newcommand{\PP}{\mathcal P}

\newcommand{\Reals}{\mathbb R}

\renewcommand{\S}{\mathbb S}

\newcommand{\IntEnt}{{\operatorname{IntEnt}}}

\newcommand{\Ent}{{\operatorname{Ent}}}
\newcommand{\proj}{{\operatorname{proj}}}
\newcommand{\Hilb}{{\operatorname{Hilb}}}

\newcommand{\lift}{{\upsilon}}
\newcommand{\resid}{{\zeta}}

\newcommand{\PW}{{\operatorname{PW}}}

\newcommand{\supp}{\operatorname{supp}}

\newcommand{\dash}{\mbox{-}}

\newcommand{\cs}{\mathrm{c}}

\renewcommand{\d}{\mathrm{d}}				

\newcommand{\oo}{\infty}
\newcommand{\cov}{{\operatorname{cov}}}

\newcommand{\true}{{\operatorname{true}}}

\newcommand{\curry}{\operatorname{curry}}

\newcommand{\p}{\mathrm p}

\renewcommand{\varsigma}{r}

\newcommand{\w}{\xi}

\newcommand{\op}{\operatorname{op}}

\numberwithin{equation}{section}
\theoremstyle{definition}
\newtheorem{thm}{Theorem}[section]
\newtheorem{cor}[thm]{Corollary}
\newtheorem{lem}[thm]{Lemma}

\newtheorem{pro}[thm]{Proposition}

\newtheorem{rem}[thm]{Remark}
\newtheorem{exa}[thm]{Example}

\renewcommand{\bar}[1]{\overline{#1}}
\renewcommand{\tilde}[1]{\widetilde{#1}}
\renewcommand{\hat}[1]{\widehat{#1}}

\renewcommand{\null}{\operatorname{null}}

\newcommand{\var}{\operatorname{var}}

\renewcommand{\op}{{\operatorname{op}}}

\renewcommand{\Re}{\operatorname{Re}}

\newcommand{\f}{\mathrm f}
\newcommand{\OLS}{{{\operatorname{OLS}}}}

\newcommand{\OLSoppara}[1]{\hat \Mech_{\OLS,#1}}
\newcommand{\OLSoptheta}{\OLSoppara{\theta}}
\newcommand{\OLSopnone}{\hat \Mech_\OLS}

\newcommand{\hatpi}{{\mathtt{\hat \pi}}}

\renewcommand{\L}{{\operatorname{L}}}
\newcommand{\R}{{\operatorname{R}}}
\newcommand{\MSE}{\operatorname{MSE}}
\newcommand{\estvar}{\operatorname{estvar}}
\newcommand{\bias}{\operatorname{bias}}

\newcommand{\notom}[1]{{}}

\author{Tom LaGatta and P. Richard Hahn}
\title{A Structural Approach to Coordinate-Free Statistics}
\date{Spring 2014}

\begin{document}

\begin{abstract}
		We consider the question of learning in general topological vector spaces. By exploiting known (or parametrized) covariance structures, our Main Theorem demonstrates that any continuous linear map corresponds to a certain isomorphism of embedded Hilbert spaces. By inverting this isomorphism and extending continuously, we construct a version of the Ordinary Least Squares estimator in absolute generality. Our Gauss-Markov theorem demonstrates that OLS is a ``best linear unbiased estimator'', extending the classical result. We construct a stochastic version of the OLS estimator, which is a continuous disintegration exactly for the class of ``uncorrelated implies independent'' (UII) measures. As a consequence, Gaussian measures always exhibit continuous disintegrations through continuous linear maps, extending a theorem of the first author. Applying this framework to some problems in machine learning, we prove a useful representation theorem for covariance tensors, and show that OLS defines a good kriging predictor for vector-valued arrays on general index spaces. We also construct a support-vector machine classifier in this setting. We hope that our article shines light on some deeper connections between probability theory, statistics and machine learning, and may serve as a point of intersection for these three communities.
		
	\end{abstract}

\maketitle

\tableofcontents


	\part{Structure Theory} \label{part1}
	
	\section{Introduction} \label{sect_intro}

		\newcommand{\data}{y}
		\newcommand{\Mech}{\Upsilon}
		\newcommand{\para}{v}
		
		\newcommand{\Data}{Y}
		\newcommand{\Para}{V}

		Suppose that we are provided with randomly sampled data of the type $\data = \Mech(\para)$, and our goal is to estimate the parameter $\para$. We assume throughout that parameters and data are sampled from topological vector spaces $\Para$ and $\Data$, respectively, and that $\Mech : \Para \to \Data$ is a (known) continuous linear map. The spaces $\Para$ and $\Data$ encode the \emph{types} of parameters and data we may work with.\footnote{By making minimal topological assumptions on spaces and mappings throughout, we ensure maximal applicability of our results.} An \emph{estimator} is a (partial) function $\hat \Mech : \Data \to \Para$ that is consistent with all possible data (i.e., $(\Mech \circ \hat \Mech)(y) = y$ for all $y$ in the domain). Throughout, we assume that the covariance structure of the random parameters is known (or hyperparametrized), in order to learn about unknown parameter values using observed data values.\footnote{A hyperparameter parametrizes a probability distribution over parameters. In practice, estimating the covariance structure empirically is a difficult second-order problem, and we do not consider it here. We treat hyperparameters as either fixed or continuously varying.}
		
		In Section \ref{sect_struct}, we describe the structure theory of probability measures on topological vector spaces (largely developed by Vakhania, Tarieladze and Chobanyan \cite{vakhania1975topological,vakhania1978covariance,vakhania1981probability,vakhania1987probability,tarieladze2007disintegration}). The basic structure is encoded in diagram \eqref{dia_basic}, which ensures that we may represent the covariance structure using Hilbert subspaces $U \subseteq V$ and $U_\Mech \subseteq Y$ (i.e., Cameron-Martin spaces). The deep structure is that $U_\Mech$ is isomorphic to a subspace $\hat U_\Mech \subseteq U$. By restricting the domain of the map $\Mech$, our Main Theorem (Theorem \ref{thm_iso}) ensures that the restriction map $\Mech : U \to U_\Mech$ corresponds to orthogonal projection onto $\hat U_D$, and therefore that $\Mech : \hat U_\Mech \to U_\Mech$ is an isomorphism of Hilbert spaces. As a consequence, the inverse map $\Mech^{-1} : U_\Mech \to \hat U_\Mech$ is well-defined and continuous. In Section \ref{sect_OLS}, we exploit this structure to construct the general Ordinary Least Squares estimator $\hat \Mech_\OLS : \Data \to \Para$, simply by defining $\hat \Mech_\OLS := \Mech^{-1}$ on the Hilbert subspace $U_D \subseteq \Data$, and extending continuously to its maximum possible domain.\footnote{Technically, the domain for the OLS estimator is the closed subspace $\bar{U_D} \subseteq Y$. If the mean vector $m \in V$ is known (or hyperparameterized), we require that the OLS estimator map the data mean $m_\Mech := \Mech(m)$ map to full mean $m$. In that case, the domain can be extended to the affine space $m_\Mech + \bar{U_D}$. If $U_D$ is the reproducing kernel Hilbert space (RKHS) for a strictly positive-definite kernel, then $U_D$ is dense in $Y$ (i.e., $\bar{U_D} = Y$).}
		
		Recall Hadamard's definition of a mathematical problem: ``a problem is well-posed if its solution (1) exists, (2) is unique, and (3) depends continuously in the observed data [and hyperparameters]'' \cite[slide 59]{poggio}. In Section \ref{sect_linreg}, we state the problem of linear regression, which is solved by the OLS estimator. Our Gauss-Markov theorem (Theorem \ref{thm_GMT}) ensures that the OLS estimator is the ``best linear unbiased estimator'' (BLUE) in general. Continuity of the OLS estimator $y \mapsto \hat \Mech_\OLS(y)$ is assured by the existence result (Theorem \ref{thm_OLS}), and joint continuity with the hyperparameters is assured under mild metrizability assumptions on the parameter and data spaces (Lemma \ref{lem_ctyshort}).\footnote{Precisely, if $V$ is a Fr\'echet space and $Y$ is a Banach space, then the function $(\theta,y) \mapsto \hat \Mech_{\OLS,\theta}(y)$ is jointly continuous.} This geometric structure is all familiar to experts in coordinate-free statistics \cite{wichura2006coordinate}, though our level of generality is novel. 
		
		In Section \ref{sect_lincond}, we state the problem of linear conditioning, which is solved by finding a continuous disintegration for the law with respect to $\Mech$. We define the ``stochastic OLS estimator'' by adding independent residual noise to the estimated value $\hat \Mech_\OLS(y)$. Our Theorem \ref{thm_UII} demonstrates that the stochastic OLS estimator defines a continuous disintegration if and only if the law of the parameters is an ``uncorrelated implies independent'' (UII) measure. Gaussian measures are the prototypical example of UII measures (and the motivation for the general definition), hence Gaussian measures always admit continuous disintegrations through continuous linear maps (Corollary \ref{cor_gaussian}).\footnote{This extends Theorems 2 \& 3 of \cite{lagatta2013continuous}, where a certain necessary and sufficient condition was found for Gaussian measures to satisfy the continuous-disintegrations property, which is now seen to be always satisfied.}
		
		\begin{rem}[Linear Models] \label{rem_linmod}
			In statistics, one often deals with linear models of the form $y = x\beta + \varepsilon$. There, $\beta$ denotes a vector-valued parameter from some space $B$, the linear operator $x : B \to Y$ encodes all explanatory variables, and $\varepsilon \in Y$ denotes additive noise (assumed to have mean zero, and be uncorrelated with $x$ and $\beta$).\footnote{In finite-dimensional applications, $y$ and $\varepsilon$ are $n$-dimensional column vectors, $\beta$ is a $p$-dimensional column vector, and $x$ is an $n \times p$ matrix, so $Y \cong \S^n$, $B \cong \S^p$ and $X \cong \S^{np}$. Independence of $x$ and $\beta$ is called ``exogeneity''.} There are two ways to formulate this in the functional form $\data = \Mech(\para)$, while preserving linearity of the map $\Mech$.
			
			First, suppose that the explanatory variables $x$ are known, and the goal is to estimate the parameter $\beta$. Construct the unified parameter value $v := (\beta, \varepsilon)$, treat $x$ as fixed, and define $\Mech_x(\beta,\varepsilon) := x\beta + \varepsilon$. Here, the parameter space is $V := B \times Y$ and the map $\Mech_x : B \times Y \to Y$ is continuous and linear. Using Theorem \ref{thm_OLS}, the OLS estimator $\hat \Mech_{\OLS,x} : Y \to B \times Y$ is well-defined.\footnote{Assuming that $\varepsilon$ can take all possible values, the domain of $\hat \Mech_{\OLS,x}$ is the full space $Y$.} When $B$ and $Y$ are both Hilbert spaces (e.g., finite-dimensional), the OLS estimator can be expressed in a closed form. Define the functions $\hat \beta_x : Y \to B$ and $\hat \varepsilon_x : Y \to Y$ by
				\begin{equation}
					\hat \beta_x(y) := \big((x^* x)^{-1} x^* \big)(y) \quad \mathrm{and} \quad \hat \varepsilon_x(y) := y - \big( x (x^* x)^{-1} x^* \big)(y) = \big( \I_Y - x\hat \beta_x \big)(y), \end{equation}
			where $x^* : Y \to B$ denotes the (Hilbert-)adjoint of $x$, and define the OLS estimator by $\hat \Mech_{\OLS,x}(y) := \big( \hat\beta_x(y), \hat \varepsilon_x(y) \big)$.\footnote{The Hilbert-adjoint $x^*$ is defined using the inner products on $Y$ and $B$, via the identity $\langle x^* y', \beta' \rangle_B = \langle y', x \beta' \rangle_Y$ for all $y' \in Y$, $\beta' \in B$. The operator $\I_Y : Y \to Y$ denotes the identity on$Y$.} The operator $\hat \beta_x := (x^* x)^{-1} x^*$ is known as the Moore-Penrose pseudo-inverse of $x$; it is always well-defined and continuous, and may be computed using the singular-value decomposition (SVD) of $x$.\footnote{As a matrix, the adjoint $x^*$ is the conjugate-transpose of $x$. If $x = u\Sigma v^*$ denotes the SVD of $x$, then $\hat \beta_x = v \hat \beta_\Sigma u^*$ \cite{ben2003generalized}.} Under mild conditions on the law of the noise (full rank covariance) and the explanatory variables (no multicolinearity), the law of large numbers implies that this estimator is \emph{consistent}: as $\dim Y \to \oo$ (in a suitable sense), the law of $\hat \beta_x(y)$ converges to a point-mass at the true value. 
			
			Second, suppose that the parameter $\beta$ is known, and the goal is to estimate the explanatory variable $x$. Construct the unified parameter value $v' := (x, \varepsilon)$, treat $\beta$ as fixed, and define $\Mech_\beta(x,\varepsilon) := x\beta + \varepsilon$. Here, the parameter space is $V' := X \times Y$ and the map $\Mech_\beta : X \times Y \to Y$ is continuous and linear. Using Theorem \ref{thm_OLS}, the OLS estimator $\hat \Mech_{\OLS,\beta} : Y \to X \times Y$ is again well-defined. Define the Koopman operator $\kappa_\beta : X \to Y$ by $\kappa_\beta(x) := x \beta$. By expressing the data in the form $y = \kappa_\beta(x) + \varepsilon$, we reduce the problem to the first case. Suppose that $X$ and $Y$ are Hilbert spaces, define the functions $\hat x_\beta : Y \to X$ and $\hat \varepsilon_\beta : Y \to Y$ by
				\begin{equation}
					\hat x_\beta(y) := \big( ( \kappa_\beta^* \kappa_\beta)^{-1} \kappa_\beta^* \big)(y) \quad \mathrm{and} \quad \hat \varepsilon_\beta(y) := y - \big( \kappa_\beta (\kappa_\beta^* \kappa_\beta)^{-1} \kappa_\beta^* \big)(y) = \big( \I_Y - \kappa_\beta \hat x_\beta \big)(y) = y - \hat x_\beta(y) \beta, \end{equation}
			where $\kappa_\beta^* : Y \to X$ denotes the (Hilbert-)adjoint Koopman operator, and define the OLS estimator by $\hat \Mech_{\OLS,\beta}(y) := \big( \hat x_\beta(y), \hat \varepsilon_x(y) \big)$.\footnote{The Hilbert-adjoint $\kappa_\beta^*$ is defined using the inner products on $Y$ and $X$, via the identity $\langle \kappa_\beta^* y', x' \rangle_X = \langle y', \kappa_\beta x' \rangle_Y = \langle y', x'\beta \rangle_X$ for all $y' \in Y$, $x' \in X$. } As before, $\hat x_\beta$ is the Moore-Penrose inverse of $\kappa_\beta$, and may be computed using the SVD. In this case, consistency is more difficult, since we need $\dim B \to \oo$ for the law of large numbers, which weakens the assumption of no-multicolinearity. Even still, the OLS estimator is well-defined and continuous.
			
			Note: if we try to simultaneously estimate $x$ and $\beta$ using a unified-parameter formalism, then we lose linearity of the map $\Mech(x,\beta,\epsilon) := x\beta + \epsilon$, owing to the product $x\beta$.
		\end{rem}
		
		In Part \ref{part3}, we apply the structure theory to some problems in machine learning, by way of spatial statistics (vector-valued arrays on arbitrary index spaces).\footnote{We require minimal topological assumptions on value and index spaces (no metrizability assumptions); in particular, index space should be separable and locally compact Hausdorff (LCH).} In Section \ref{sect_spatialstats}, we state the problem of prediction, which is to predict all values of a random array given only the values on a restricted index set. Our Theorem \ref{thm_covrep} ensures that covariance structures may always be represented using covariance tensors in the vector-valued setting. Using an arbitrary covariance tensor, we construct the OLS estimator for general array mappings (Theorem \ref{thm_covrep}), which yields the kriging predictor as an immediate consequence (Corollary \ref{cor_kriging}).\footnote{Optimality of these estimators is ensured by the Gauss-Markov theorem.} Our work generalizes the current state-of-the-art kriging predictor, which requires Hilbert-valued arrays on finite-dimensional index spaces \cite{menafoglio2013universal}.
		
		In Section \ref{sect_SVM}, we state the problem of classification, which is to partition the index space given sets of labeled data. The solution is provided by the support-vector machines (SVM) algorithm of Cortes and Vapnik \cite{cortes1995support}, which reduces the classification problem to a convex optimization problem in Hilbert space. We present a general version of the SVM classifier for arbitrary index spaces, highlighting the role played by topological compactness (rather than finiteness). 
		
		
		Our article demonstrates that a structural approach is a valuable way to formulate and solve problems in coordinate-free statistics. Our Main Theorem ensures that the only essential metric structures are those on Cameron-Martin spaces, and these are a derived consequence of underlying covariance structure. Consequently, unnecessary topological assumptions (like bases and metrics) may safely be removed from spaces and mappings, increasing the range of applicability of mathematical results. Our general estimators and classifiers, being constructed in a formal, mathematical language, may readily be converted into functional programs.\footnote{e.g., using a language like Haskell, Clojure or R.} Finally, we hope that tools like estimators and classifiers may prove to be as useful for the mathematics community as they have been for those in statistics and machine learning. \newline	
		
				\textbf{Acknowledgements.}  The authors would like to thank Scott Armstrong, Antonio Auffinger, Ivan Barrientos, Michael Betancourt, Tyler Bryson, Elliot Aguilar, Miguel Carri\'on Alv\'arez, Adam Brandenburger, Emily Chambliss, Brian D'Alessandro, John Dawkins, Serina Diniega, Anthony Di Franco, Levent Do\u{g}u\c{s} Sa\u{g}un, Nate Eldridge, Lindsey K. Gamard, David Glickenstein, Michael Greinecker, Owen Haaga, Adrien Hardy, Joey Hirsh, Aukosh Jagannath, Selin Kalaycioglu, Samantha Kappagoda, Gustavo Lacerda, S. Dan Lovell, Robert Kohn, Arjun Krishnan, Yenming Mark Lai, Elliot Lipnowski, Peter McCullagh, Casey Meakin, Bud Mishra, Mehryar Mohri, David Mordecai, Charles M. Newman, Victor de la Pe\~{n}a, Natesh Pillai, Benjamin Pittman-Polletta, Phillip Protter, Andy Putman, Javier Rodriguez Laguna, Daniel Roy, Silvia N. Santalla, Jonathan Schmidt-Domin\'e, David Spivak, Daniel Stein, John Terilla, Timothy Ter\"av\"ainen, J. Tipan Verella, Joseph Watkins, Janek Wehr, Brad Weir, Joe Wells, Jochen Wengenroth, Lai-Sang Young, Tai Young-Taft and Fran\c cois Ziegler.
		
		T.L.'s research and travel supported in part by NSF PIRE grant OISE-07-30136. 

	\section{Structure Theory of Probability Measures on Topological Vector Spaces} \label{sect_struct}
				
		
		Let $\Para$ denote a topological vector space, representing the parameter space for our system of interest. i.e., $\Para$ is a vector space (over field of scalars $\S := \Reals$ or $\C$), it is equipped with a topology, and the operations of vector addition and scalar multiplication are both continuous with respect to this topology. We make minimal topological assumptions on $\Para$, namely, that it is complete and Hausdorff and its dual space separates points.\footnote{Completeness means that Cauchy nets (and sequences) converge and Hausdorff means that points may be separated by closed sets.} The dual space $V^*$ consists of continuous linear functionals $f : \Para \to \S$; separating points means that for any $v,v' \in V$, there exists a continuous linear functional $f \in \Para^*$ so that $\Re f[\para] < \Re f[\para']$. The minimality of these assumptions ensures that nearly \emph{any} topological vector space is potentially usable in statistical modeling.\footnote{For a long list of non-trivial topological vector spaces, see \cite{khaleelulla1982counterexamples}.} 
		
		For example, if the system is described by $N$ scalar parameters, then $\Para = \S^N \cong C(\N, \S)$. If the system is described by continuous (scalar-valued) time series, then $\Para = C(\Reals, \S)$. If the system is described by a vector-valued time series, then $\Para = C(\Reals, V_0)$, where $V_0$ denotes some other topological vector space of values.\footnote{Discontinuous time series may be studied using Skorokhod space $\Para = D(\Reals, V_0)$ \cite{whitt2002stochastic}. In that case, $V_0$ should be metrizable (e.g., Fr\'echet).} If the system is described by a vector-valued array (indexed by some space $I$), then $\Para = C(I, V_0)$. All these examples are ``array spaces'', which will be discussed more in Part \ref{part3}. A linear model is described by $V = X \times Y$ or $V' = B \times Y$ (as in Remark \ref{rem_linmod}); general linear models permit the spaces $B$, $X$ and $Y$ to be array spaces.

		\subsection{Random Parameters} \label{sect_structspaces}
		
		We suppose that the parameter vector for the system is random, and its law is to be described by some (Radon) probability measure on $V$.\footnote{We recall the basic definition. Let $\B(V)$ denote the Borel $\sigma$-algebra of $V$. A (Borel) probability measure is a countably-additive function $P : \B(V) \to [0,1]$ satisfying the constraints $P(\varnothing) = 0$, $P(V) = 1$ and $P(A \cup B) = P(A) + P(B) - P(A \cap B)$ for any $A,B \in \B(X)$. A Radon probability measure is also inner regular, meaning that probabilities of open sets may be approximated by probabilities of upward converging compact sets. i.e., $P(A) = \lim_\gamma P(K^\gamma)$ when $K^\gamma \uparrow A$. For an introduction to measure theory, see Folland \cite{folland1999real} or Bogachev \cite{bogachev2007measure}.} We ignore the issue of whether there exists a single ``true'' law $\P_\true$ for the system, and consider only statistical models: (hyper-)parameterized families of probability measures $\{ \P_\theta \}$ on parameter space $V$ (cf. McCullagh \cite{mccullagh2002statistical}). 
		
		We make minimal topological assumptions on statistical models: hyperparameter space $\Theta$ should be topological, and the model function $\theta \mapsto \P_\theta$ be continuous (in the topology of weak convergence of measures).\footnote{i.e., if $\theta^\gamma \to \theta$ is a convergent net (or sequence) of hyperparameters, then the measures $\P_{\theta^\gamma}$ are tight and converge weakly to $\P_\theta$ \cite{billingsley2008probability}.} Beyond this, we are indifferent to the nature or interpretation of the parametrization: the space $\Theta$ could range from a single point ($\Theta = \{ \true \}$ or $\{ \theta_0\}$) to the full simplex of all (Radon) probability measures on $V$ ($\Theta = \Delta(V)$). If $\Theta$ is a topological manifold (as it often is in applications), then it can be equipped with a Fisher-Rao information metric, encoding statistical structure of the (hyper-)parameterization \cite{amari2007methods}.

			\begin{rem}
				In the linear modeling setting ($y = x\beta + \varepsilon$), this corresponds to a family of measures $\{\P_\theta\}$ on the joint parameter space, $V := B \times Y$, which allows us to control the joint distribution of the parameter-and-noise. Our perspective is Bayesian; the data, noise and parameters are all assumed to be stochastic. The parametrization $\theta \mapsto \P_\theta$ allows us to further hedge our uncertainty, without committing us to a fixed hyperparameter.
			\end{rem}
		
		We make two structural assumptions on the statistical model. First, $\P$ should satisfy the \emph{separable-support hypothesis}, meaning that support set $\supp \P_\theta$ is separable (has a countable dense subset) and has full measure in $V$.\footnote{i.e., $\P_\theta\big( \supp \P_\theta \big) = 1$ for all $\theta$. Recall that the support is the intersection of all closed sets of full measure. If the support set is non-separable, then it may not exist, may equal the empty set, or may not have full measure \cite{vakhania1975topological}.} In practice, this is not a strict assumption, since separability seems to be a necessary precursor for computability \cite{roy2011computability}. If parameter space $\Para$ is separable, then the separable-support hypothesis is automatically satisfied.
		
		We also assume that $\P$ satisfies the \emph{finite-variance hypothesis}, which states that every continuous linear functional has finite variance: $\var_\theta(f) < \oo$ for all $f \in \Para^*$ and $\theta \in \Theta$.\footnote{Let $\E_\theta$ denote the expectation operator of $\P_\theta$, defined using the Lebesgue integral: $\E_\theta[s] := \int_\Para s(\para) \, \P_\theta(\d \para)$ for any $\P_\theta$-integrable $s : V \to \S$. The variance is defined by $\var_\theta(s) := \E_\theta[|s|^2] - \big| \E_\theta[s] \big|^2$.} In practice, many data are heavy-tailed (and infinite-variance), so care should be taken with statistical modeling in such contexts (cf. Remark \ref{rem_infvar}).
				
		Under these two assumptions, every measure $\P_\theta$ admits a \emph{covariance structure}. The key representation theorem is \cite[Theorem 3.ii]{vakhania1978covariance}, which ensures that $\P_\theta$ admits a mean vector $m_\theta \in \Para$ and a covariance operator $k_\theta : \Para^* \to \Para$.\footnote{If $\P$ is finite variance and non-separable-support, then the covariance structure of $\P_\theta$ can be represented by a covariance operator under stronger conditions on the space $V$ (e.g., nuclearity) \cite[Theorem 3.(iii--vi)]{vakhania1978covariance}.} The mean vector is the Pettis integral of $\P_\theta$, in the sense that $\E_\theta[f] = f[m_\theta]$ for all functionals $f \in \Para^*$. The covariance operator $k_\theta : \Para^* \to \Para$ is a symmetric, nonnegative-definite, continuous linear operator which satisfies $\cov_\theta[e | f] = \bar e[ k_\theta f]$ for all $f, e \in \Para^*$.\footnote{The symmetry and nonnegative-definite conditions are necessary since the covariance defines a (pre-)inner product on $\Para$. i.e., $\bar e[k_\theta f] = \cov_\theta[e|f] = \bar{\cov_\theta[f|e]} = \bar f[k_\theta e]$ and $\bar f[k f] = \var_\theta(f) \ge 0$ for all $f, e \in \Para^*$. We follow the physicists' convention that the covariance is anti-linear in the first component, and linear in the second component. Every symmetric, nonnegative-definite continuous linear operator is the covariance operator for some measure on $\Para$ (possibly non-Gaussian) \cite{vakhania1987probability}. There is a long history in trying to classify operators which generate Gaussian measures; we refer the reader to \cite{vakhania1981probability, ledoux1991probability}.} Model continuity ensures that the mean and covariance parametrizations $\theta \mapsto m_\theta$ and $\theta \mapsto k_\theta$ are continuous.\footnote{i.e., if $\theta^\gamma \mapsto \theta$ is a convergent net (or sequence) of hyperparameters, then $m_{\theta^\gamma} \to m_\theta$ in $\Para$, and for every functional $f \in \Para^*$, $k_{\theta^\gamma} f \to \bar k_\theta f$ in $\Para$. The set-valued map $\theta \mapsto \supp \P_\theta$ may not be continuous in any meaningful sense.} Vakhania's theorem \cite{vakhania1975topological} ensures that
			\begin{equation} \label{eqn_vakthm}
				\supp \P_\theta \subseteq m_\theta + \bar{k_\theta \Para^*}; \end{equation}
		we refer to the space $m_\theta + \bar{k_\theta \Para^*}$ as the \emph{affine support} of measure $\P_\theta$. Identity \eqref{eqn_vakthm} allows us to reduce statistical questions about $\P_\theta$ to geometric questions about the space $m_\theta + \bar{k_\theta \Para^*}$.\footnote{For a proof of \eqref{eqn_vakthm}, see the proof of (2.4) of \cite{lagatta2013continuous}.}
		
		The covariance defines a (pre-)inner product on the dual space $\Para^*$, and therefore may be completed into a (separable) Hilbert space $H_\theta$.\footnote{When $\dim V = \oo$, the weak topology on $V^*$ will be distinct from the covariance inner-product topology. Formally, $H_\theta$ is the Hilbert-space completion of the quotient space $\Para^* / \null k_\theta$, where $\null k_\theta = \{ f : \var_\theta(f) = 0 \}$ is the null space of the covariance operator. Note that $\null k_\theta$ consists of those functionals which equal zero almost everywhere. Separability of $H_\theta$ is a consequence of separable support.} The next lemma demonstrates that the covariance operator must factor through this Hilbert space, and ensures that $H_\theta$ is isomorphic to a Hilbert subspace $U_\theta \subseteq \Para$, called the (embedded) Cameron-Martin space for $\P_\theta$. The spaces $H_\theta$ and $U_\theta$ completely encode the covariance structure of $\P_\theta$. We define the affine Cameron-Martin space by translating by the mean: $A_\theta := m_\theta + U_\theta$. Clearly, $A_\theta$ is dense in the affine support $\bar{k_\theta \Para^*} =: \bar{A_\theta}$.\footnote{If $\dim H_\theta = \oo$, then $\P_\theta(U_\theta) = 0 = \P_\theta(A_\theta)$ \cite[Theorem 1.3]{bell1987malliavin}.}
				
		
			\begin{lem}[Three-Space Diagram] \label{lem_threespace}
				There exists a continuous dense map $\iota_\theta^* : \Para^* \to H_\theta$ and a continuous injective map $\iota_\theta : H_\theta \hookrightarrow \Para$ making the following diagram commute:
					\begin{equation} \label{dia_classic}
						\begin{matrix} \xymatrix{ \Para^* \ar@{->}[rr]^{k_\theta} \ar@{->>}[rd]_{\iota_\theta^*} && \Para \\ & H_\theta \ar@{^(->}[ur]_{\iota_\theta} & } \end{matrix} \end{equation}
				Equivalently, the covariance operator admits the factorization $k_\theta = \iota_\theta \circ \iota_\theta^*$.\footnote{In infinite-dimensions, the dense map $\iota_\theta^*$ is not surjective. The maps $\iota_\theta^*$ and $\iota_\theta$ depend only on the covariance operator $k_\theta$, and no other features of the measure $\P_\theta$.}
 The operator $\iota_\theta^*$ is adjoint to $\iota_\theta$, meaning that $\cov_\theta\!\big[ \iota_\theta^* f \big| h \big] = \bar f\big[ \iota_\theta h \big]$ for all $f \in \Para^*$ and $h \in H_\theta$. The Cameron-Martin space is defined by $U_\theta := \iota_\theta H_\theta$.
			\end{lem}
		
			The proof of Lemma \ref{lem_threespace} can be found in Appendix \ref{lem_threespaceproof}. \newline
			
			This construction is familiar in the probability literature. The embedding $H_\theta \hookrightarrow \Para$ was originally discovered by Cameron and Martin \cite{cameron1944transformations} for the special case of Brownian motion. Gross \cite{gross1965abstract} generalized the construction in his formulation of abstract Wiener spaces, as a general way to construct Gaussian measures on separable Banach spaces. Dudley, Feldman \& Le Cam \cite{dudley1971seminorms} demonstrated that this is the only way to construct Gaussian measures. In spatial statistics ($V = C(I,V_0)$), the Cameron-Martin space $U_\theta$ corresponds to the reproducing kernel Hilbert space (RKHS) for a covariance kernel on $I$, and $k_\theta$ corresponds to the integral operator (Section \ref{sect_kriging}). Embedding the index space into Hilbert space is known as the ``kernel trick'' in machine learning (Section \ref{sect_SVM}).
			
			There is a second natural embedding, $\PW_\theta : H_\theta \hookrightarrow L^2(V,\S; \P_\theta)$, called the Paley-Wiener map.\footnote{The $L^2$-function $v \mapsto \PW_\theta(h)(v)$ is called the Paley-Wiener integral of $h$.} The Paley-Wiener map is defined on linear functionals by subtracting the mean, $\PW_\theta(\varphi)(v) := \varphi[v] - \E_\theta[\varphi]$, then extending continuously to all of Hilbert space. Clearly, this is an isometry. Since $H_\theta \cong U_\theta$, the Paley-Wiener map is well-defined on the space $U_\theta$. The Cameron-Martin theorem \cite{cameron1944transformations} provides a formula for how a Gaussian measure transforms (under translation by $u$) in terms of the Paley-Wiener integral $v \mapsto \PW_\theta(u)(v)$.\footnote{The Girsanov theorem \cite{girsanov1960transforming} deals with the special case of Brownian motion (i.e., Wiener measure), and the theorem by Lenglart \cite{lenglart1977transformation} deals with the general case of semimartingales. For a recent introduction, see Protter \cite{protter2004stochastic}.}

			\begin{rem}[Infinite-Variance] \label{rem_infvar}
				In the case of infinite-variance, finite $p$th moment (for $p \in (1,2)$), the structure theory is considerably weaker. Theorem 3.ii of \cite{vakhania1978covariance} still ensures that $V^* \hookrightarrow L_p(V,\S;\P_\theta)$, and dually, $L_{p^*}(V, \S; \P_\theta) \subseteq V$. Unlike the Hilbert case, there is no map from $L_p$ to $L_{p^*}$, which prevents us from closing the diagram (as in \eqref{dia_classic}). We leave this direction open for future research.\footnote{The case of finite 1st moment (but infinite $p$th moment for $p >1$) is extremely difficult, and there are few rigorous results.}
			\end{rem}

		\subsection{Random Data} \label{sect_structmaps}

		Suppose that $Y$ is another data space, satisfying the same basic assumptions on $\Para$ as in the previous section.\footnote{We assume that $\Para$ and $Y$ are complete and Hausdorff, and their dual spaces separate points.} Let $\Mech : \Para \to Y$ denote a continuous linear map. As discussed in the introduction, many problems in statistics can be formulated using continuous linear maps. 
		
				
		For each hyperparameter $\theta$, define the push-forward measure $\P_{\Mech,\theta} := \Mech_* \P_\theta := \P_\theta \circ \Mech^{-1}$.\footnote{i.e., $\P_{\Mech,\theta}(B) = \P_\theta\big( \Mech^{-1}B \big)$ for all $B \in \B(Y)$.} The expectation operator $\E_{\Mech,\theta}[\cdot]$ is defined using the change-of-variables formula.\footnote{i.e., $\E_{\Mech,\theta}[s] := \int_Y s(y) \, \P_{\Mech,\theta}(\d y) = \int_\Para s\big( \Mech(\para) \big) \, \P_\theta(\d \para)$ for any integrable function $s : Y \to \S$.} Since $\Mech : \Para \to Y$ is a continuous linear function, the \emph{adjoint map} $\Mech^* : Y^* \to \Para^*$ is also continuous and linear.\footnote{The adjoint generalizes the transpose of a matrix, and is defined by the formula $(\Mech^* f)[\para] := (f \circ \Mech)[\para] = f\big[ \Mech(\para) \big]$ for all $\para \in \Para$. i.e., the adjoint $\Mech^*$ acts on a linear functional by pre-composing with $\Mech$. Adjoint operators are known as ``Koopman operators'' in the literature on dynamical systems \cite{budivsic2012applied}. When the space $\Para$ is finite dimensional, the adjoint map $\Mech^*$ is also called the transpose of $\Mech$. No metric structure on either $\Para$ or $Y$ is necessary to define $\Mech^*$: the adjoint map is purely topological and algebraic.} The model $\P_\Mech : \Theta \to \Delta(Y)$ (defined by $\theta \mapsto \P_{\Mech,\theta}$) satisfies the finite-variance and separable-support hypotheses.\footnote{To see this, suppose that $\theta^\gamma \to \theta$ is a convergent net (or sequence) in $\Theta$, and let $s : Y \to \S$ be a continuous, bounded function. Consequently, $\E_{\Mech,\theta^\gamma}[s] = \int_\Para (s \circ \Mech)(\para) \, \P_{\theta^\gamma}(\d \para) \to \int_\Para (s \circ \Mech)(\para) \, \P_\theta(\d \para) = \E_{\Mech,\theta}[s]$ since the composition $s \circ \Mech$ is continuous and bounded, and the measure-valued function $\theta \mapsto \P_\theta$ is continuous. If $e \in Y^*$, then $\var_{\Mech,\theta}(e) = \var_\theta(\Mech^* e) < \oo$, so $\P_\Mech$ is finite-variance.} 
		
		The push-forward measure $\P_{\Mech,\theta}$ has mean vector $m_{\Mech,\theta} := \Mech(m_\theta) \in Y$ and covariance operator $\Mech k_\theta \Mech^* : Y^* \to Y$.\footnote{i.e., $\Mech k_\theta \Mech^* := \Mech \circ k_\theta \circ \Mech^*$. Composition of continuous linear maps generalizes matrix multiplication.} The affine support of $\P_{\Mech,\theta}$ equals $m_{\Mech,\theta} + \bar{\Mech k_\theta \Mech^* Y^*}$, and Vakhania's theorem \eqref{eqn_vakthm} implies that $\supp \P_{\Mech,\theta} \subseteq m_{\Mech,\theta} + \bar{\Mech k_\theta \Mech^* Y^*}$.
		
		As in Section \ref{sect_structspaces}, define the Hilbert spaces $H_{\Mech,\theta} := \bar{Y^* / \null \Mech k_\theta \Mech^*}$ for each $\theta$.\footnote{The space $H_{\Mech,\theta}$ depends only on the covariance operator $k_\theta$, the continuous linear map $\Mech$ and its adjoint $\Mech^*$.} Applying Lemma \ref{lem_threespace}, there exist continuous linear maps $\iota_{\Mech,\theta}^* : Y^* \twoheadrightarrow H_{\Mech,\theta}$ and $\iota_{\Mech,\theta} : H_{\Mech,\theta} \hookrightarrow Y$ so that $\Mech k_\theta \Mech^* = \iota_{\Mech,\theta} \circ \iota_{\Mech,\theta}^*$.
		
		The next lemma demonstrates that the abstract Cameron-Martin space $H_{\Mech,\theta}$ is isomorphic to a closed subspace $\hat H_{\Mech,\theta} \subseteq H_\theta$, called the \emph{lifted Cameron-Martin space}. Formally, the lifted space $\hat H_{\Mech,\theta}$ is defined as the closure of $(\iota_\theta^* \circ \Mech) Y^*$ in $H_\theta$. 
		
			\begin{lem}[Six-Space Diagram] \label{lem_twospaces}
				There exist a surjective map $\pi_{\Mech,\theta} : H_\theta \twoheadrightarrow H_{\Mech,\theta}$ and an isomorphism $\eta_{\Mech,\theta} : H_{\Mech,\theta} \hookrightarrow \hat H_{\Mech,\theta}$ which together satisfy the identities
					\begin{equation} \label{eqn_twoidentities}
						\pi_{\Mech,\theta} \circ \iota_\theta^* \circ \Mech^* = \iota_{\Mech,\theta}^* \qquad \mathrm{and} \qquad \Mech \circ \iota_\theta \circ \eta_{\Mech,\theta} = \iota_{\Mech,\theta}. \end{equation}
				Consequently, the following diagram commutes for all paths beginning at $Y^*$ or $H_{\Mech,\theta}$, as well as all paths ending at $H_{\Mech,\theta}$ or $Y$:
					\begin{equation} \label{dia_basic}
						\begin{matrix} \xymatrix{ Y^* \ar@{->}[r]^{\Mech^*} \ar@{->>}[rrdd]_{\iota_{\Mech,\theta}^*} & \Para^* \ar@{->}[rr]^{k_\theta} \ar@{->>}[rd]^{\iota_\theta^*} && \Para \ar@{->}[r]^\Mech & Y \\ && H_\theta \ar@{^(->}[ru]^{\iota_\theta} \ar@/_/@{->>}[d] && \\ && H_{\Mech,\theta}  \ar@{^(->}[rruu]_{\iota_{\Mech,\theta}}  \ar@/_/@{^(->}[u] &&} \end{matrix} \end{equation}
				The unlabeled arrows denote the maps $\pi_{\Mech,\theta}$ and $\eta_{\Mech,\theta}$, respectively. 
				
				
			\end{lem}
			
			The proof of Lemma \ref{lem_twospaces} can be found in Appendix \ref{lem_twospacesproof}. 
		
		\subsection{Main Theorem}
		
		We define the embedded Cameron-Martin space $U_{\Mech,\theta} := \iota_{\Mech,\theta} H_{\Mech,\theta}$ and the affine Cameron-Martin space $A_{\Mech,\theta} := m_{\Mech,\theta} + U_{\Mech,\theta}$, both in $Y$. Similarly, we define the lifted spaces $\hat U_{\Mech,\theta} := \iota_\theta \hat H_{\Mech,\theta}$ and $\hat A_{\Mech,\theta} := m_\theta + \hat U_{\Mech,\theta}$, both in $\Para$. Lemma \ref{lem_threespace} implies that the Hilbert spaces $\hat U_{\Mech,\theta}$ and $\hat A_{\Mech,\theta}$ are isomorphic to $U_{\Mech,\theta}$ and $A_{\Mech,\theta}$, respectively. Our Main Theorem states that these isomorphisms are realized by the original map $\Mech : \Para \to Y$, restricted to the domains $\hat U_{\Mech,\theta}$ and $\hat A_{\Mech,\theta}$. This is a remarkable fact, since the function $\Mech$ may be neither surjective nor injective on the whole space; even if the inverse exists, it may not be continuous on its domain. 
		
		Let $\hat U_{\Mech,\theta}^\perp$ denote the orthogonal complement of $\hat U_{\Mech,\theta}$ in $U_\theta$. Similarly, let $\hat A_{\Mech,\theta}^\perp$ denote the orthogonal complement of $\hat A_{\Mech,\theta}$ in $A_\theta$. Clearly, $\hat A_{\Mech,\theta}^\perp = m_\theta + \hat U_{\Mech,\theta}^\perp$. The Main Theorem also states that the action of $\Mech$ on these spaces is trivial.


			\begin{thm}[Main Theorem] \label{thm_iso}
				The restriction maps $\Mech : \hat U_{\Mech,\theta} \to U_{\Mech,\theta}$ and $\Mech : \hat A_{\Mech,\theta} \to A_{\Mech,\theta}$ are isomorphisms of Hilbert spaces. In particular, the inverse maps $\Mech^{-1}_\theta : U_{\Mech,\theta} \to \hat U_{\Mech,\theta}$ and $\Mech^{-1}_\theta : A_{\Mech,\theta} \to \hat A_{\Mech,\theta}$ are well-defined and continuous.\footnote{The continuity of the inverse maps is with respect to both the subspace and Hilbert topologies on $U_{\Mech,\theta}$ and $A_{\Mech,\theta}$. Note that the inverse of $m_{\Mech, \theta}$ is fixed to be $m_\theta$, even though the fiber $\Mech^{-1}(m_{\Mech, \theta})$ may possess multiple elements. The inverse maps $\Mech^{-1}_\theta$ depend on the hyperparameter. \label{foot_Lthm1}} 
				
				The (less) restricted maps $\Mech : U_\theta \to U_{\Mech,\theta}$ and $\Mech : A_\theta \to A_{\Mech,\theta}$ are surjective. In the first case, the kernel of $\Mech|_{U_\theta}$ equals $\hat U_{\Mech,\theta}^\perp$. In the second case, $\big(\Mech|_{A_\theta}\big)^{-1}\big( m_{\Mech,\theta} \big) = \hat A_{\Mech,\theta}^\perp$. 
			\end{thm}
			\begin{proof}
				By Lemma \ref{lem_twospaces}, we have that $\Mech \circ \iota_\theta \circ \eta_{\Mech,\theta} = \iota_{\Mech,\theta}$ on the domain $H_{\Mech,\theta}$. Consequently, $\Mech = \iota_{\Mech,\theta} \circ \pi_{\Mech,\theta} \circ \iota_\theta^{-1}$ on the domain $\hat U_{\Mech,\theta}$. Since the three maps $\iota_\theta^{-1} : \hat U_{\Mech,\theta} \to \hat H_{\Mech,\theta}$, $\pi_{\Mech,\theta} : \hat H_{\Mech,\theta} \to H_{\Mech,\theta}$ and $\iota_{\Mech,\theta} : H_{\Mech,\theta} \to U_{\Mech,\theta}$ are Hilbert-space isomorphisms, their composition $\Mech : \hat U_{\Mech,\theta} \to U_{\Mech,\theta}$ is also a Hilbert-space isomorphism.\footnote{The maps $\iota_\theta^{-1}$ and $\pi_{\Mech,\theta}$ may not be isomorphisms on their unrestricted domains.} Consequently, the inverse map $\Mech^{-1} : U_{\Mech,\theta} \to \hat U_{\Mech,\theta}$ is also a Hilbert-space isomorphism, hence well-defined and continuous.\footnote{Alternatively, the continuity of $\Mech^{-1}$ follows from the Banach-Schauder theorem: the map $\Mech : \hat U_{\Mech,\theta} \to U_{\Mech,\theta}$ is a surjective map of Hilbert spaces, hence is an open map.} It easily follows that the kernel of $\Mech : U_\theta \to U_{\Mech,\theta}$ is the space $\hat U_{\Mech,\theta}^\perp$.
				
				The corresponding claims for the affine spaces follow by shifting by the mean. 
			\end{proof}
		
			In \cite{lagatta2013continuous}, LaGatta found a necessarily and sufficient condition for a Gaussian measure to admit a continuous disintegration with respect to $\Mech$. That condition corresponded to continuity of the inverse map $\Mech^{-1} : U_{\Mech,\theta} \to \Para$; indeed, the Main Theorem ensures that this condition is always satisfied. Adapting the argument of \cite[Theorem 3]{lagatta2013continuous}, Corollary \ref{cor_gaussian} follows \emph{mutatis mutandis}.\footnote{The results in \cite{lagatta2013continuous} are stated for a Banach space $\Para$. The more general Corollary \ref{cor_gaussian} follows from our Theorem \ref{thm_UII}.}
			
			\begin{cor} \label{cor_gaussian}
				Gaussian measures always admit continuous disintegrations through continuous linear maps.
			\end{cor}

		\section{The Ordinary Least Squares Estimator} \label{sect_OLS}
		
		Using the Main Theorem, it is straightforward to construct the Ordinary Least Squares (OLS) estimator. Recall that an \emph{estimator} of $\Mech : \Para \to Y$ is a (partial) function $\hat \Mech : Y \to \Para$ which is a right-inverse of $\Mech$. This condition means that $(\Mech \circ \hat \Mech)(y) = y$ for any possible data $y \in \supp \P_{\Mech,\theta}$, ensuring coherency between the estimate and the data.\footnote{A right-inverse is known as a \emph{section} in category theory \cite{mac1998categories}.}
		
		Fix a hyperparameter $\theta$, and consider the isomorphic subspaces $\hat A_{\Mech,\theta} \subseteq V$ and $A_{\Mech,\theta} \subseteq Y$. By the Main Theorem, the isomorphism is realized by the restricted map $\Mech : \hat A_{\Mech,\theta} \to A_{\Mech,\theta}$. Therefore, the inverse map $\Mech^{-1} : A_{\Mech,\theta} \to \hat A_{\Mech,\theta}$ is well-defined and continuous. The OLS estimator $\OLSoptheta: \bar{A_{\Mech,\theta}} \to \bar{A_{\Mech,\theta}}$ is simply defined to be the continuous extension of $\Mech^{-1}$ to the closed affine support. Note that we have imposed the condition $\OLSoptheta(m_{\Mech,\theta}) = m_\theta$.

		
		\begin{thm}[Existence of the OLS Estimator] \label{thm_OLS}
			The Ordinary Least Squares estimator $\OLSoptheta: \supp \P_{\Mech,\theta} \to \supp \P_\theta$ is a well-defined, continuous linear estimator. Furthermore:
				\begin{enumerate}
					\item \label{thm_OLS2} The OLS estimator is unbiased, and maps the mean in $Y$ to the mean in $\Para$. i.e., $\OLSoptheta(m_{\Mech,\theta}) = m_\theta$.
					\item \label{thm_OLS4} Estimation is contravariant (i.e., reverses the order of composition). Consider sequential continuous linear maps $\Mech^1 : V \to Y$, $\Mech^2 : Y \to Z$ and $\Mech^{2,1} : V \to Z$ with $\Mech^{2,1} := \Mech^2 \circ \Mech^1$. Then $\OLSoptheta^{2,1} = \OLSoptheta^1 \circ \OLSoptheta^2$.

				\end{enumerate}

			\end{thm}
			
			The proof of Theorem \ref{thm_OLS} may be found in Appendix \ref{thm_OLSproof}. \newline
			
			
		Theorem \ref{thm_OLS} ensures that the OLS estimator is continuous, which is important across all applications. Under some (mild) metrizability assumptions on the spaces $\Para$ and $\Data$, we can in fact prove a stronger continuity result: the OLS estimator varies jointly continuously in both the data and the hyperparameters. 
		
			\begin{lem}[Strong Continuity Lemma, short version] \label{lem_ctyshort}
				Suppose that $\Para$ is a Fr\'echet space and $\Data$ is a Banach space. Then the joint OLS estimator $(\theta,y) \mapsto \OLSoptheta(y)$ is jointly continuous.\footnote{The joint domain is the extended parameter space $\Theta_f := \{ (\theta,y) : y \in \supp \P_{\Mech,\theta} \}$, equipped with the subspace topology inherited from $\Theta \times Y$.}
			\end{lem}
		
			The long version and proof of the Strong Continuity Lemma may be found in Appendix \ref{app_cty}.

		\part{Some Problems in Statistics} \label{part2}
		
		We now apply the structure theory to some problems in statistics. In Section \ref{sect_linreg}, we demonstrate that the OLS estimator solves the problem of linear regression. In Section \ref{sect_lincond}, we construct the stochastic OLS estimator, which solves the problem of linear conditioning for ``uncorrelated implies independent'' (UII) measures.

		\section{The Problem of Linear Regression and the General Gauss-Markov Theorem} \label{sect_linreg}
		
		The problem of linear regression is to optimally estimate an unknown parameter $v$ given the observed data $y = \Mech(v)$. The optimality condition is to simultaneously minimize estimated variance and mean-squared error (both defined with respect to the parametrized covariance structure of the law of $v$). The Gauss-Markov theorem (Theorem \ref{thm_GMT}) resolves this problem: the OLS estimator is optimal among the class of unbiased linear estimators.


		

		\subsection{Auxiliary Operators} \label{sect_auxops}

		Since the OLS estimator $\OLSoptheta: \supp \P_{\Mech,\theta} \to \supp \P_\theta$ is well-defined and continuous, its adjoint operator $\OLSoptheta^* : \Para^* \to (\supp \P_{\Mech,\theta})^*$ is also well-defined and continuous.\footnote{The adjoint estimator acts by pre-composing a continuous linear functional by $\OLSoptheta$. That is, $(\OLSoptheta^* f)[y] := f[ \OLSoptheta y]$ for all $y \in \bar{L_{\Mech,\theta}}$ and $f \in \Para^*$. Conveniently, the adjoint is defined on the maximal domain $V^*$.}
		
		Let $\hat \Mech : \supp \P_{\Mech,\theta} \to \supp \P_\theta$ denote an arbitrary continuous linear estimator.\footnote{Being an estimator means that $\Mech \circ \hat \Mech$ is the identity on $\supp \P_{\Mech,\theta}$.} We define the lifted estimator by $\L_{\hat \Mech}(v) := (\hat \Mech \circ \Mech)(v)$, which gives the result of estimation by $\hat \Mech$ after observing the value $\Mech(y)$. We define the residual estimator $\R_{\hat \Mech}(v) := v - \L_{\hat \Mech}(v)$, which gives the displacement of this estimation from the original value $v$.\footnote{Both $\L_{\hat \Mech}, \R_{\hat \Mech} : \bar{A_\theta} \to \bar{A_\theta}$ are continuous linear operators defined on the closed affine support $\bar{A_\theta} = m_\theta + \bar{U_\theta} = m_\theta + \bar{k_\theta V^*}$, which contains the support $\supp \P_\theta$.}
		
		The Gauss-Markov theorem relies on the following geometric lemma, which states that any estimator corresponds to an oblique projection in the relevant Hilbert spaces, and orthogonality is achieved only in the case of the OLS estimator. Optimality of orthogonal projection corresponds to statistical optimality of Ordinary Least Squares.
		
			\begin{lem} \label{lem_ortho}
				Any lifted estimator $\L_{\hat \Mech}$ is the continuous extension of an oblique projection $A_\theta \to \hat A_{\Mech, \theta}$, and its residual estimator $\R_{\hat \Mech}$ is the corresponding extension of the complementary projection $A_\theta \to \hat A_{\Mech, \theta}^\perp$. These projections are orthogonal if and only if $\hat \Mech = \OLSoptheta$ is the OLS estimator.
			\end{lem}
			\begin{proof}
				Idempotence is assured since any estimator is a right-inverse. i.e., $\L_{\hat \Mech}^2 = \hat \Mech \circ \big( \Mech \circ \hat \Mech \big) \circ \Mech = \hat \Mech \circ \Mech = \L_{\hat \Mech}$. Consequently, $\R_{\hat \Mech}^2 = \big( \I_\Para - \L_{\hat \Mech}\big)^2 = \I_\Para - 2 \L_{\hat \Mech} + \L_{\hat \Mech}^2 = \I_\Para - \L_{\hat \Mech} = \R_{\hat \Mech}$. It is straightforward to verify that $\L_{\hat \Mech}$ maps $A_\theta$ into $\hat A_{\Mech,\theta}$. By the Main Theorem, the OLS estimator extends orthogonal projection, since $\big(\Mech|_{A_\theta}\big)^{-1}\big( m_{\Mech,\theta} \big) = \hat A_{\Mech,\theta}^\perp$.
			\end{proof}

		Since the estimator $\hat \Mech$ is continuous and linear (by assumption), then so are the lifted estimators $\L_{\hat \Mech}$ and $\R_{\hat \Mech}$. Consequently, the adjoint auxiliary operators $\L_{\hat \Mech}^* := \Mech^* \circ \hat \Mech^*$ and $\R_{\hat \Mech}^* = \I_{\Para^*} - \L_{\hat \Mech}^*$ are well-defined and continuous.\footnote{The adjoined lifted estimators $\L_{\hat \Mech}^*$ and $\R_{\hat \Mech}^*$ are functions $\Para^* \to (\supp \P_\theta)^*$. Pointwise, these adjoints are defined by $(\L_{\hat \Mech}^* f)[\para] = f[L_{\hat \Mech} \para] = f[ (\hat \Mech  \circ \Mech)\para] = ((\Mech^* \circ \hat \Mech^*) f)[\para]$ and $(\R_{\hat \Mech}^* f)[\para] = f[\para] - (\L_{\hat \Mech}^* f)[\para]$ for all $\para \in \supp \P_\theta$ and $f \in \Para^*$.} These also correspond to oblique projections in the dual of the Hilbert space. These adjoint projections equal the original projections only in the case of Ordinary Least Squares.
		
		\subsection{The Generalized Gauss-Markov Theorem} \label{sect_GMT}
		
		The \emph{estimated variance} of a functional $f$ (with respect to $\hat \Mech$) is the variance of $\L_{\hat \Mech}^* f$ (equivalently, $\hat \Mech^* f$). The \emph{mean-squared error} of $f$ is the $L^2$-norm of $\R_{\hat \Mech}^* f$. That is,
			\begin{equation} \label{def_estvarMSE}
				\estvar_{\hat \Mech,\theta}(f) := \var_\theta\!\big( \L_{\hat \Mech}^* f\big) = \var_{\Mech,\theta}\!\big( \hat \Mech^* f \big) \quad \mathrm{and} \quad \MSE_{\hat \Mech,\theta}(f) := \E_\theta\!\big[\big| \R_{\hat \Mech}^* f \big|^2 \big] \end{equation}
		The estimated variance is the amount of variance ``explained'' using the estimator $\hat \Mech$. The mean-squared error is the residual variance after estimation. 
				
		 The bias of a functional $f$ (with respect to $\hat \Mech$) is defined by $\bias_{\hat \Mech,\theta}(f) := - \E_\theta[\R_{\hat \Mech}^* f] = -f[\R_{\hat \Mech}m_\theta]$, and is the (negative) average residual vector. An estimator is unbiased if and only if $\bias_{\hat \Mech,\theta}(f) = 0$ for all $f \in \Para^*$. Theorem \ref{thm_OLS}.\ref{thm_OLS2} implies that the OLS estimator $\OLSoptheta$ is unbiased. A trivial consequence of the definitions is the \emph{bias-variance tradeoff}, which states that \begin{equation} \label{eqn_biasvariance}
						\MSE_{\hat \Mech, \theta}(f) = \big| \bias_{\hat \Mech,\theta}(f) \big|^2 + \estvar_{\hat \Mech,\theta}(f) + \Big( \var_\theta(f) - 2 \Re\!\big( \cov_\theta[\L_{\hat \Mech}^* f | f ] \big) \Big) \end{equation}
				for all functionals $f \in \Para^*$.\footnote{There are versions of the bias-variance tradeoff for arbitrary loss functions \cite{hastie2009elements}; equation \eqref{eqn_biasvariance} corresponds to the loss function being mean-squared error. See \cite{domingos2000unified} for a general formula for various loss functions, and \cite{munson2009feature} for a nice application of bias-variance tradeoff to feature selection.}
				


		

			\begin{thm}[Generalized Gauss-Markov Theorem] \label{thm_GMT}
				The OLS estimator is the ``best linear unbiased estimator'' (BLUE), minimizing both expected variance and mean-squared error. This means that, if $\hat \Mech$ is any continuous linear estimator (possibly biased), then
					\begin{equation} \label{eqn_GMT}
						\estvar_{\OLS,\theta}(f) \le \estvar_{\hat \Mech, \theta}(f) \qquad \mathrm{and} \qquad \MSE_{\OLS,\theta}(f) \le \MSE_{\hat \Mech,\theta}(f) - \big| \bias_{\hat \Mech,\theta}(f) \big|^2 \end{equation}
				for any functional $f \in \Para^*$, with equality if and only if $\hat \Mech^* f = \OLSoptheta^* f$. 
			\end{thm}
			\begin{proof}
				By Lemma \ref{lem_ortho}, a continuous linear estimator corresponds to orthogonal projection if and only if it equals the OLS estimator. Orthogonal projection is optimal, and minimizes both the covariance distance projected onto $\hat A_{\Mech,\theta}$, as well as onto the complement $\hat A_{\Mech,\theta}^\perp$. Estimated variance and mean-squared error extend these distances to arbitrary functionals, which proves the inequalities \eqref{eqn_GMT}. Saturation occurs if and only if the oblique projection agrees with orthogonal projection in the direction of the functional $f$.
				
			\end{proof}
			
		\newcommand{\JS}{{\operatorname{JS}}}
		
		It is easy to see that the OLS estimator is inadmissible (not optimal) when compared with nonlinear estimators. If one is estimating three (or more) functionals simultaneously, then the (nonlinear) James-Stein estimator $\hat \Mech_{\JS}$ \cite{james1961estimation} improves total mean-squared error.\footnote{The James-Stein estimator can be improved by considering the positive-part James-Stein estimator $\hat \Mech_{\JS^+}$ \cite{anderson1954introduction}, which is itself inadmissible owing to a certain smoothness constraint \cite{lehmann1998theory}.} Nonetheless, OLS is a workhorse of statistics, and is sufficient for many practical purposes.

		\section{The Problem of Linear Conditioning and Stochastic Estimators} \label{sect_lincond}
		
		An estimator $\hat \Mech : Y \to X$ is a powerful tool, as it selects a single plausible data source $\hat v = \hat \Mech(y)$ given an observation $y = \Mech(\para)$. However, this functional consistency is also a drawback, as there is no room for uncertainty in the prediction. To accommodate uncertainty, we define a \emph{stochastic estimator} of $\Mech$ (with respect to $\P_\theta$) to be any measure-valued function $y \mapsto \PP_{\theta|\Mech=y}$ which satisfies the coherency and continuity constraints:
			\begin{enumerate}
				\item (Coherency) For all possible data $y \in \P_{\Mech,\theta}$, the measure $\PP_{\theta|\Mech=y}$ is supported on the fiber $\Mech^{-1}(y) \subseteq V$.
				\item (Continuity) The function $y \mapsto \PP_{\theta|\Mech=y}$ is weakly continuous.\footnote{i.e., if $y^\gamma \to y$ is a convergent net (or sequence), then the measures $\PP_{\theta|\Mech=y^\gamma}$ converge weakly to $\PP_{\theta|\Mech=y}$. The continuity constraint ensures that, for all $B \in \B(V)$, the real-valued function $y \mapsto \PP_{\theta|\Mech=y}(B)$ is (Borel-)measurable.}
			\end{enumerate}
		The continuity condition ensures that probability estimates are robust to perturbations in the data. A stronger condition is joint continuity in the data and hyperparameters, meaning that $(\theta, y) \mapsto \PP_{\theta|\Mech=y}$ is jointly continuous.\footnote{i.e., if $(\theta^\gamma, y^\gamma) \to (\theta, y)$, then $\PP_{\theta^\gamma|\Mech=y^\gamma} \to \PP_{\theta|\Mech=y}$ weakly.} 
		
		In Section \ref{sect_stochest}, we demonstrate how every estimator $\hat \Mech$ can be used to construct a stochastic estimator $\PP_{\hat \Mech, \theta| \Mech=y}$: estimate the value $\hat \Mech(y)$, and add independent residual noise.\footnote{Conversely, every stochastic estimator defines an estimator, where we set $\hat \Mech(y)$ to be the mean vector of the measure $\PP_{\theta|\Mech=y}$.} Our Proposition \ref{pro_PPstrongcty} demonstrates that, under mild metrizability assumptions, the stochastic OLS estimator $(\theta, y) \mapsto \PP_{\OLS,\theta|\Mech=y}$ is jointly continuous.
		
		\subsection{The Problem of Conditioning} \label{sect_history}
		
		The problem of conditioning is to construct a (Bayes) optimal stochastic estimator, i.e., one whose statistics agree with conditional probabilities. This means that optimal stochastic estimators are exactly \emph{continuous disintegrations}, and satisfy the disintegration equation 
			\begin{equation} \label{eqn_disint}
				\int_\Para s(\para) \, \P_\theta(\d \para) = \int_Y \int_\Para s(\para) \, \PP_{\theta|\Mech=y}(\d \para) \P_{\Mech,\theta}(\d y), \end{equation}
		for every integrable $s : \Para \to \Data$. By the disintegration theorem \cite{leao2004regular,durrett2010probability}, measurable disintegrations exist in wide generality, but such abstract arguments fail to ensure continuity.\footnote{Note: if $\PP_{\theta|\Mech=y}$ is a continuous disintegration and $\P_{\theta|\Mech=y}$ is a measurable disintegration, then $\PP_{\theta|\Mech=y} = \P_{\theta|\Mech=y}$ for $\P_{\Mech,\theta}$-almost every $y$. A continuous disintegration captures topological information, which a purely measurable disintegration may miss.}
	
		In Section \ref{sect_contdis}, we demonstrate that any stochastic estimator defines a continuous disintegration of a certain ``convolution measure''. If the convolution measure agrees with the original measure, then trivially, the stochastic estimator is seen to be optimal. In Section \ref{sect_UII}, we discuss the class of ``uncorrelated implies independent'' measures, which includes the class of Gaussian measures as the prototypical case. Our Theorem \ref{thm_UII} demonstrates that the UII condition is equivalent to the measure equaling its OLS convolution measure. Consequently, the stochastic OLS estimator is optimal exactly for the class of UII measures, hence Gaussian measures always admit continuous disintegrations through continuous linear maps (Corollary \ref{cor_gaussian}). This completely resolves the problem of conditioning for UII measures and continuous linear maps. 		
		
			\begin{rem}
				Tjur \cite[Theorem 8.1]{tjur1975constructive} studied the problem of nonlinear conditioning for a smooth, nonlinear map $\Mech : V \to Y$ between finite dimensional manifolds $V$ and $Y$. His arguments relied on the existence of Lebesgue measure (the universal reference measure on finite-dimensional, Euclidean $Y$), and did not generalize to the setting of infinite-dimensional $Y$. As he wrote, ``this problem of conditioning on a stochastic process is too complicated to be dealt with here'' \cite[p.~18]{tjur1975constructive}. 
		
				Vakhania and Tarieladze \cite[Theorem 3.11]{tarieladze2007disintegration} studied the problem of linear conditioning for a continuous linear map $\Mech : V \to Y$ between Banach spaces, but they could not surmount the obstacle of infinite-dimensional data. They required that the map $\Mech$ have finite rank (in particular, finite-dimensional support). Their ``average-case optimal algorithm'' is a special case of the OLS estimator, and their Proposition 4.1 is a version of the generalized Gauss-Markov theorem.
		
				LaGatta \cite[Theorem 2]{lagatta2013continuous} essentially resolved the problem of linear conditioning for Gaussian measures on Banach spaces, without any hypothesis of finite rank.\footnote{LaGatta and Wehr \cite[Theorem 12.5]{lagatta2014geodesics} replicated LaGatta's work in the setting of a Fr\'echet space $\Para$ and Banach space $Y$.} To ensure continuity of the OLS estimator, LaGatta required a certain necessary and sufficient condition to define (``$M < \oo$''), corresponding to the continuity of the inverse map $\Mech^{-1} : A_{\Mech,\theta} \to \hat A_{\Mech,\theta}$. Our Main Theorem (Theorem \ref{thm_iso}) demonstrates that this condition is spurious: the inverse map is always well-defined and continuous. 
			\end{rem}

		For a structural approach to continuous disintegrations using category theory and groupoids, see Censor and Grandini \cite{censor2010borel}. For applications of disintegrations to statistics, see Chang and Pollard \cite{chang1997conditioning}.

		
		



		\subsection{Constructing Stochastic Estimators} \label{sect_stochest}
		
		Let $\Mech : V \to Y$ be a continuous linear map, and let $\hat \Mech : Y \to V$ be an arbitrary continuous linear estimator (e.g., the OLS estimator $\hat \Mech_{\OLS,\theta}$). We use this estimator to define a certain stochastic estimator $y \mapsto \PP_{\hat \Mech, \theta| \Mech=y}$. This measure is centered on the estimated value $\hat \Mech(y)$, and samples are displaced from this value by the addition of residual noise. This measure is unbiased exactly if $\hat \Mech(m_{\Mech,\theta}) = m_\theta$.
		
		To formalize this construction, we define two auxiliary measures. Let $\L_{\hat \Mech} := \hat \Mech \circ \Mech$ and $\R_{\hat \Mech} := \I_V - \L_{\hat \Mech}$ denote the lifted and residual estimators (as in Section \ref{sect_auxops}). We define the auxiliary measures by pushing the original measure through these operators. The lifted measure $\P_{\hat \Mech,\theta} := \big( \L_{\hat \Mech} \big)_* \P_\theta := \P_\theta \circ \L_{\hat \Mech}^{-1}$ is an isomorphic copy of the data measure $\P_{\Mech, \theta}$, but supported in parameter space $V$. The residual measure $\P_{\hat \Mech, \theta}^\perp := \big( \R_{\hat \Mech} \big)_* \P_\theta := \P_\theta \circ \R_{\hat \Mech}^{-1}$ represents the distribution of residual noise.
		
			\begin{lem}
				The lifted measure $\P_{\hat \Mech,\theta}$ has mean vector $\hat \Mech(m_{\Mech,\theta})$ and covariance operator $k_{\hat \Mech,\theta} := \L_{\hat \Mech} k_\theta \L_{\hat \Mech}^* = \L_{\hat \Mech} k_\theta = k_\theta \L_{\hat \Mech}^*$. The residual measure $\P_{\hat \Mech, \theta}^\perp$ has mean vector $m_\theta - \hat \Mech(m_{\Mech,\theta})$ and covariance operator $k_{\hat \Mech, \theta} := \R_{\hat \Mech} k_\theta \R_{\hat \Mech}^* = \R_{\hat \Mech} k_\theta = k_\theta \R_{\hat \Mech}^*$. 
			\end{lem}
			\begin{proof}
				By Lemma \ref{lem_ortho}, the lifted estimator corresponds to an oblique projection. Since projections are idempotent (and Hilbert space is self-dual), the different covariance representations are seen to hold.
			\end{proof}
		
		Fix a data value $y$. To construct the measure $\PP_{\hat \Mech, \theta|\Mech=y}$, we simply translate the residual measure $\P_{\hat \Mech, \theta}^\perp$ by the estimated value $\hat \Mech(y)$. Formally, if $\tau_{\hat \Mech(y)}(v) := \hat \Mech(y) + v$ denotes the translation-by-$\hat \Mech(y)$ operator, then we define 
			\begin{equation} \label{def_stochest}
				\PP_{\hat \Mech, \theta| \Mech=y} := \big( \tau_{\hat \Mech(y)} \big)_* \P_{\hat \Mech, \theta}^\perp = \big( \tau_{\hat \Mech(y)} \big)_* \big( \R_{\hat \Mech} \big)_* \P_\theta = \big( \tau_{\hat \Mech(y)} \circ \R_{\hat \Mech} \big)_* \P_\theta. \end{equation}
		Equivalently, if $s : V \to \S$ is any integrable function, then
			\begin{equation} \label{def_stochestint}
				\int_\Para s(v) \, \PP_{\hat \Mech, \theta| \Mech=y}(\d v) = \int_\Para s\big( \hat \Mech(y) + \resid\big) \, \P_{\hat \Mech,\theta}^\perp(\d \resid). \end{equation}

		\begin{lem} \label{lem_PPcty}
			Let $\hat \Mech$ be any continuous estimator. The measure-valued function $y \mapsto \PP_{\hat \Mech, \theta|\Mech=y}$ satisfies the coherency and continuity constraints, hence is a stochastic estimator.\footnote{i.e., each $\PP_{\hat \Mech,\theta|\Mech=y}$ is supported on $\Mech^{-1}(y)$, and $y \mapsto \PP_{\theta|\Mech=y}$ varies continuously. The domain of $y \mapsto \PP_{\hat \Mech,\theta|\Mech=y}$ is the same as $\operatorname{dom} \hat \Mech$, which should contain $\supp \P_{\Mech,\theta}$ as a subset. Lemma \ref{lem_PPcty} works for biased, non-linear estimators in addition to unbiased, linear estimators. If $\hat \Mech$ is measurable (but not continuous), then $y \mapsto \PP_{\hat \Mech, \theta|\Mech=y}$ is measurable (but not continuous).}
			\end{lem}
			
			The proof of Lemma \ref{lem_PPcty} can be found in Appendix \ref{lem_PPctyproof}. \newline
		
		\subsection{The Stochastic OLS Estimator} \label{sect_stochOLSest}
		
		Let $\hat \Mech_{\OLS,\theta}$ denote the OLS estimator. Define the lifted and residual OLS estimators by $\L_{\Mech,\OLS,\theta} := \hat \Mech_{\OLS,\theta}$ and $\R_{\Mech,\OLS,\theta} := \I_V - \L_{\Mech,\OLS,\theta}$, and the lifted and residual OLS measures by $\P_{\Mech,\OLS,\theta} := \big( \L_{\Mech,\OLS,\theta} \big)_* \P_\theta$ and $\P_{\Mech,\OLS,\theta}^\perp := \big( \R_{\Mech,\OLS,\theta} \big)_* \P_\theta$. Define the stochastic OLS estimator by $\PP_{\OLS,\theta|\Mech=y} := \big( \tau_{\hat \Mech_{\OLS,\theta}}\big)_* \big( \R_{\Mech,\OLS,\theta} \big)_* \P_\theta$. 
		
		Theorem \ref{thm_UII} demonstrates that the stochastic OLS estimator is optimal only in the case for ``uncorrelated implies independent'' measures. Even if stochastic OLS does not define a continuous disintegration, next result ensures that the mean and covariance structure of $\P_{\theta|\Mech=y}$ is guaranteed to be correct for $\P_{\Mech,\theta}$-almost every $y$. If the higher moments are not a concern in a given application domain, then the stochastic OLS estimation is satisfactory for most practical purposes.
		
			\begin{pro} \label{pro_samemeancov}
				Let $y \mapsto \P_{\theta|\Mech=y}$ be a disintegration of $\P_\theta$ with respect to $\Mech$. For $\P_{\Mech,\theta}$-almost every $y$, the mean and covariance of the stochastic OLS estimator $\PP_{\OLS,\theta|\Mech=y}$ agrees with that of $\P_{\theta|\Mech=y}$.
			\end{pro}
			
		The proof of Proposition \ref{pro_samemeancov} may be found in Appendix \ref{pro_samemeancovproof}.\newline
		
		The next result ensures joint continuity of the stochastic OLS estimator under a mild metrizability assumption on the parameter and data spaces. 
		
			\begin{pro} \label{pro_PPstrongcty}
				Suppose $\Para$ and $Y$ are both Banach spaces. The joint stochastic OLS estimator $(\theta,y) \mapsto \PP_{\OLS,\theta|\Mech=y}$ is jointly continuous.\footnote{i.e., if $(\theta^\gamma, y^\gamma)$ is a net (or sequence) which converges to $(\theta,y)$ in $\Theta_f$, then the net (or sequence) of measures $\PP_{\OLS,\theta^\gamma|\Mech=y^\gamma}$ is tight and converges weakly to $\PP_{\theta|\Mech=y}$.}
			\end{pro}

			The proof of Proposition \ref{pro_PPstrongcty} can be found in Appendix \ref{pro_PPstrongctyproof}.

		\subsection{Convolution Measures and Continuous Disintegrations} \label{sect_contdis}
			
		In general, it is difficult to verify whether a stochastic estimator defines a continuous disintegration: this involves checking the disintegration equation \eqref{eqn_disint} for every integrable function. However, it is straightforward to construct a certain measure for which the stochastic estimator is \emph{trivially} a continuous disintegration. As a consequence, checking optimality of a stochastic estimator reduces to verifying whether a measure equals the corresponding convolution measure. 
		
		Given an estimator $\hat \Mech$, we define its convolution measure by the formula $\P_{\hat \Mech,\theta}^* := \P_{\hat \Mech,\theta}^\perp * \P_{\hat \Mech, \theta}$. Equivalently, for all integrable $s : \Para \to \S$,
			\begin{equation} \label{def_convmeas}
				\int_\Para s(\para) \, \P_{\hat \Mech,\theta}^*(\d \para) := \int_\Para \int_\Para s\big( \lift + \resid \big) \, \P_{\hat \Mech,\theta}^\perp(\d \resid) \, \P_{\hat \Mech,\theta}(\d \lift) = \int_\Data \int_\Para s\big( \hat \Mech(y) + \resid \big) \, \P_{\hat \Mech,\theta}^\perp(\d \resid) \, \P_{\Mech,\theta}(\d y). \end{equation}

		The convolution measure corresponds to the following stochastic algorithm. Sample a random $y$ with respect to $\P_{\Mech,\theta}$, and estimate $\hat \Mech(y)$. Sample an \emph{independent} random $v'$ with respect to $\P_{\theta}$, and compute the residual $\resid := v' - \hat \Mech\big( \Mech(v') \big)$. Then the sum $v := \hat \Mech(y) + \resid$ has law $\P_{\hat \Mech,\theta}^*$.\footnote{Note the importance of independence here. Suppose instead that $y$ and $v'$ are maximally dependent, i.e., $y = \Mech(v')$. In this case, the residual noise would be exactly the difference $\resid := v' - \hat \Mech(y)$. Consequently, $v := \hat \Mech(y) + \resid = v'$, so $v$ would have law $\P_\theta$ (rather than $\P_{\hat \Mech,\theta}^*$).}


		
			\begin{thm} \label{thm_convdisint}
				The stochastic estimator $y \mapsto \PP_{\hat \Mech, \theta|\Mech=y}$ always defines a continuous disintegration of the convolution measure $\P_{\hat \Mech,\theta}^*$ (with respect to $\Mech$). The disintegration equation states that 
					\begin{equation} \label{eqn_convdisint}
						\int_\Para s(\para) \, \P_{\hat \Mech,\theta}^*(\d \para) = \int_\Data \int_\Para s(\para) \, \PP_{\hat \Mech,\theta|\Mech=y}(\d \para) \, \P_{\Mech,\theta}(\d y) \end{equation}
				for all integrable $s : \Para \to \S$. Consequently, the stochastic estimator $y \mapsto \PP_{\hat \Mech, \theta|\Mech=y}$ defines a continuous disintegration of the original measure if and only if $\P_\theta = \P_{\hat \Mech,\theta}^*$.
			\end{thm}
			\begin{proof}
				The latter equality in the definition \eqref{def_convmeas} is justified by the definition of the lifted measure, $\P_{\hat \Mech,\theta} = \big( \L_{\hat \Mech} \big)_* \P_{\theta} = \hat \Mech_* \Mech_* \P_\theta = \hat \Mech_* \P_{\Mech,\theta}$. Using definition \eqref{def_stochestint}, the latter expression of \eqref{def_convmeas} is replaced by \\ $\int_\Data \int_\Para s(\para) \, \PP_{\hat \Mech,\theta|\Mech=y}(\d \para) \, \P_{\Mech,\theta}(\d y)$, which proves \eqref{eqn_convdisint}.
				
				If $\P_\theta \ne \P_{\hat \Mech,\theta}^*$, then there exists an integrable $s$ for which $\int_\Para s(\para) \, \P_\theta(\d \para)$ does not equal $\int_\Para s(\para) \, \P_{\hat \Mech,\theta}^*(\d \para) = \int_Y \int_\Para s(\para) \, \PP_{\hat \Mech,\theta|\Mech=y}(\d \para) \, \P_{\Mech,\theta}(\d y)$, using \eqref{eqn_convdisint}. 
			\end{proof}

		\subsection{UII Measures and the Stochastic OLS Estimator} \label{sect_UII}
	
		We say that a measure satisfies the ``uncorrelated implies independent'' (UII) hypothesis when its dependence structure is functionally characterized by its covariance structure. Equivalently, the following statement is satisfied for all functionals $f \in \Para^*$:
			\begin{equation} \label{eqn_UIIhyp}
				\mbox{if $f$ is $\Mech$-uncorrelated, then $f$ is $\Mech$-independent.} \end{equation}
		Gaussian measures are the archetypical examples of UII measures, and are the motivation for the general definition. Another example of a UII measure is the law of a family of associated random variables \cite[Corollary 3]{newman1984asymptotic}.\footnote{The basic ingredient in the proof that associated random variables are UII is Lebowitz's basic distribution function inequality \cite{lebowitz1972bounds}.} Since multivariate extreme-value distributions are associated \cite[Prop.~5.1]{marshall1983domains}, their laws form an important class of non-Gaussian UII measures.
		
		 Let $\P_{\Mech,\OLS,\theta}^* := \P_{\Mech,\OLS,\theta} * \P_{\Mech,\OLS,\theta}^\perp$ denote the OLS convolution measure. The next theorem demonstrates that a measure is UII if and only if it equals its OLS convolution measure. In light of Theorem \ref{thm_convdisint}, this is equivalent to the stochastic OLS estimator being a continuous disintegration. The intuition behind the proof is that the independence structure of a UII measure reduces to verifying that functionals are uncorrelated.

			\begin{thm} \label{thm_UII}
				$\P_\theta$ is a UII measure (with respect to $\Mech$) if and only if $\P_\theta = \P_{\Mech,\OLS,\theta}^*$ if and only if the stochastic OLS estimator defines a continuous disintegration. Consequently, UII measures always admit continuous disintegrations through continuous linear maps.
			\end{thm}
			
			The proof of Theorem \ref{thm_UII} may be found in Appendix \ref{thm_UIIproof}. \newline

		This completely resolves the problem of linear conditioning for UII measures. Corollary \ref{cor_gaussian} immediately follows: Gaussian measures always admit continuous disintegrations through continuous linear maps.



		\part{Some Problems in Machine Learning} \label{part3}
		
		We now apply the structure theory to some problems in machine learning. In Section \ref{sect_spatialstats}, we solve the problem of prediction using the OLS kriging predictor; we also prove a general representation theorem for covariance tensors for vector-valued arrays on arbitrary index spaces. In Section \ref{sect_SVM}, we solve the problem of classification using a general support-vector machines classifier.
		
		\section{The Problem of Prediction and the OLS Kriging Predictor} \label{sect_spatialstats}
		
		Many problems in machine learning can be recast as problems in spatial statistics, i.e., the study of random associative arrays (a.k.a. key-value pairings).\footnote{Associative arrays are a flexible structure for describing data of many different types: relational databases (like SQL) are just indexed families of tabular arrays, and graph databases (like neo4j) are just indexed families of key-value pairings (associative arrays).} A general solution is provided by ``kriging predictors'', as developed by the geostatistician Danie Krige in his master's thesis \cite{krige1951statistical}, and formulated theoretically by Georges Matheron  \cite{matheron1963principles}.

		Let $I$ denote a topological space of \emph{indexes} (or \emph{keys}), and let $V$ denote a topological vector space of primitive \emph{values}. We assume that index space $I$ is separable and locally compact Hausdorff (LCH), and that value space $V$ is locally convex, complete and Hausdorff.\footnote{The topology of a locally convex space is generated by seminorms.} Formally, a $V$-valued associative array is a (continuous) function $a : I \to V$; its set of key-value pairs is simply the graph $\big\{ \big( i, a(i) \big) \big\} \subseteq I \times V$. Under these minimal topological assumptions, the space of all arrays $V^I := C(I,V)$ is itself a topological vector space (equipped with the compact-open topology).\footnote{Precisely: $V^I$ is locally convex, complete and Hausdorff.} A \emph{random} associative array is described by a probability measure $\P$ on array space $V^I$. This satisfies the assumptions of Section \ref{sect_structspaces}, so the structure theory of this article applies. 
		
		\begin{rem}[Arrays of Arrays] \label{rem_notation}
			Let $I$ and $J$ be index spaces, and let $V$ be a (Fr\'echet) value space. An array-of-arrays can be thought of as either a $V$-valued array (indexed by the product space $I \times J$) or a $V^J$-valued array (indexed by $I$). Since $V$ is Fr\'echet, its topology is compactly generated \cite{52749}; consequently, the exponential law $V^{I \times J} \cong (V^J)^I$ is guaranteed to hold.\footnote{When $V$ is Fr\'echet, the curry operator $\curry : V^{I \times I} \to (V^J)^I$, defined by $\curry(a)(i)(j) := a(i,j)$, is a linear homeomorphism. If $V$ is not Fr\'echet, then the inverse of the curry operator may not be continuous. An example of a non-Fr\'echet value space is the (cartesian) product space $V = \prod_{r \in \Reals} \Reals$, which is much larger than the (Fr\'echet) space $\Reals^\Reals := C(\Reals, \Reals)$ \cite{164681}.}
		\end{rem}
		
		\begin{exa}[Time Series and Spatio-Temporal Statistics]
			A \emph{time series} is an associative array $a : T \to V$, where the index set is a closed subset $T \subseteq \Reals$ called ``time''. A \emph{spatial field} is an associative array $a : S \to V$, where the index set is a closed subset $S \subseteq \Reals^3$ called ``space''. A \emph{spatio-temporal field} is an associative array $a : S \times T \to V$, where the index set is ``space-time'' $S \times T \subseteq \Reals^{3+1}$.\footnote{For a comprehensive introduction to spatio-temporal statistics, see Cressie and Wikle \cite[Chapter 6]{cressie2011statistics}. For discontinuous time series, one works with Skorokhod space $D(T, V)$, consisting of right-continuous, left-limited (c\`adl\`ag) time series with values in a Fr\'echet space \cite{whitt2002stochastic}.} In light of Remark \ref{rem_notation}, we may consider spatio-temporal fields as time-varying spatial fields $(V^S)^T$, as spatially-varying time series $(V^T)^S$, or as fields which vary jointly in space and time $V^{S \times T}$. 
		
		In most applications, the value space $V$ is finite-dimensional, and each dimension corresponds to a particular attribute of interest. e.g., in hydrological applications, one considers $V = \Reals^3$, corresponding to the Hydraulic Conductivity, Porosity, and Tortuosity of a porous medium. 
		\end{exa}
		
		Because of the indexing, the covariance structure of a random array admits a representation using a \emph{covariance kernel}. In the scalar-valued case ($V = \S$), the covariance kernel is a (continuous) function $c : I \times I \to \S$ which is symmetric and non-negative definite.\footnote{Symmetry means that $c(i,i') = \bar{c(i',i)}$ for all $i,i' \in I$, and non-negative-definite means that \begin{equation} \label{def_nonneg} \sum_n \bar{s_n} s_{n'} c(i_n, i_{n'}) \ge 0 \end{equation} for any points $s_n \in \S$ and $i_n \in I$. By Kolmogorov's extension theorem \cite{karatzas1991brownian}, condition \eqref{def_nonneg} is equivalent to the condition $\int_I \int_I c(i,i') \bar{\psi}(\d i) \varphi(\d i) \ge 0$ for all ``co-arrays'', i.e., compactly supported Radon measures on $I$ (cf. Section \ref{sect_coarrays}).} In the vector-valued case, the situation is more subtle: the covariance kernel is actually a \emph{covariance tensor} $c : I \times I \to V \otimes_\cov V$ (Theorem \ref{thm_covrep}).\footnote{In the vector-valued case, the non-negative-definite condition \eqref{def_nonneg} is replaced by $\sum_n \bar{s_n} s_{n'} (\bar{e_n} \otimes e_{n'}) \big[ c(i_n, i_{n'}) \big] \ge 0$ for all $s_n \in \S$, $i_n \in I$ and $e_n \in V^*$.} We will discuss the scalar-valued case in Section \ref{sect_covfuns}, and the vector-valued case in Section \ref{sect_covtens}. The particular tensor product $V \otimes_\cov V$ is inherited from the covariance structure.
		

		\subsection{The Problem of Prediction} \label{sect_kriging}
		
		The problem of \emph{prediction} is to estimate values of an array $a : I \to V$ given only the values over a restricted subset $D \subseteq I$, and can be reformulated as the problem of linear regression from statistics.\footnote{Note: the individual arrays $a : I \to V$ are highly nonlinear objects (we do not even assume linear structure on $I$. The space $V^I$, however, is linear, since the sum of two arrays is an array, as is any scalar multiple.} Consider the two array spaces $V^I$ and $V^D$, and consider the restriction map $\Mech : V^I \to V^D$. i.e., $\Mech(a)(d) := a(d)$ for all $a \in V^I$ and $d \in D$. A \emph{kriging predictor} is a continuous function $\hat \Mech : V^D \to V^I$ which is consistent with the observed data. i.e., $\hat \Mech(y)(d) = y(d)$ for all $y \in V^D$ and $d \in D \subseteq I$.\footnote{In practice, the words ``predictor'' and ``estimator'' are largely used as synonyms. Estimators estimate parameters, and predictors predict unseen values. In this section, the unseen values are represented as a parameter array $a : I \to V$.} 
		
		By exploiting the structure theory, we may use OLS to construct a natural kriging predictor $\hat \Mech_{\OLS,c} : V^D \to V^I$ for any fixed covariance kernel $c$.\footnote{We suppose that $c$ is strictly positive-definite, so that the OLS estimator is well-defined for all possible array data $y \in V^D$.} The Gauss-Markov theorem (Theorem \ref{thm_GMT}) ensures that $\hat \Mech_{\OLS,c}$ is a good kriging predictor, minimizing both estimated variance and mean-squared error of predicted data values. When $V$ is a Banach space and $D$ is compact, the Strong Continuity Lemma (Lemma \ref{lem_ctyshort}) ensures that the operator $(c,y) \mapsto \hat \Mech_{\OLS,c}(y)$ is jointly continuous in the kernel and the data. We organize this as the following result, which follows from the more general Theorem \ref{thm_OLSarrays}. 
		
		\begin{cor}[OLS Kriging Predictor] \label{cor_kriging}
				Suppose that $D$ is a closed subset of $I$, and let $\Mech : V^I \to V^D$ be the restriction map, defined by $\Mech(a)(d) = a(d)$ for all $d \in D$. If $c$ is a strictly positive-definite kernel on $I$, then the OLS kriging predictor $\hat \Mech_{\OLS,c} : V^D \to V^I$ is well-defined and continuous on the maximal domain. The OLS kriging predictor is a ``best linear unbiased predictor'' (BLUP). The following diagram commutes:
				\begin{equation} \label{dia_2ma}
					\begin{matrix} \xymatrix{ & V^I \ar@{->}@/^/[rr]^{\Mech} && V^D \ar@{_(-->}@/^/[ll]^{\hat \Mech_{\OLS,c}} \\ U_c \ar@{^(->}[ur] \ar@{->>}@/^/[rr]^{\Mech} && U_{D,c} \ar@{_(->}@/^/[ll]^{\Mech^{-1}_c} \ar@{^(->}[ur] & } \end{matrix} \end{equation}	
				If $V$ is a Banach space and $D$ is compact, then $(c,y) \mapsto \hat \Mech_{\OLS,c}(y)$ is jointly continuous.
			\end{cor}

		When both $I$ and $D$ are finite, the well-definedness of the OLS estimator is a straightforward exercise \cite{wichura2006coordinate}. When $I$ is infinite and $D$ is finite, the ``average-case optimal algorithm'' of Vakhania and Tarieladze \cite{tarieladze2007disintegration} is exactly the OLS estimator. When $I$ and $D$ are both infinite and $V$ is a Banach space, the OLS estimator corresponds to the map $m$ of \cite{lagatta2013continuous}. Our construction generalizes the current state-of-the-art kriging predictor \cite{menafoglio2013universal}, which is defined for finite-dimensional $I$ and Hilbert space $V$.
		
		
		The kernel should be chosen to ensure certain qualitative properties of the estimated arrays. For example, if $I$ is a smooth manifold, the data are $C^k$-smooth, and the covariance function is $C^{2k+1}$-smooth, then the predicted kriging functions will also be $C^k$-smooth. If prediction is to have no effect beyond some threshold, compactly covariance functions may be chosen \cite{gneiting2002compactly}.


		\subsection{Co-Arrays} \label{sect_coarrays}
		
				
		We now present the structure theory of random arrays. Let $V^*$ denote the dual space of value space $V$, consisting of ``co-values'' (continuous linear functionals $f : V \to \S$). Let  $(V^I)^*$ denote the dual space of array space $V^I$, consisting of ``co-arrays'' (continuous linear functionals $\varphi : V^I \to \S$). The Riesz representation theorem \cite[Chapter 4, page 156]{bourbaki2004integration} ensures that co-array space $(V^I)^*$ is isomorphic to $M_\cs(I, V^*)$, the space of compactly supported, $V^*$-valued Radon measures on $I$.\footnote{Recall that a $V^*$-valued Radon measure is a countably-additive function $\varphi : \B(I) \to V^*$ satisfying $\varphi(\varnothing) = 0$ and the inclusion-exclusion principle $\varphi(B \cup B') = \varphi(B) + \varphi(B') - \varphi(B \cap B')$ for all $B,B' \in \B(I)$. If index space $I$ has infinite cardinality, then the dual space $(V^I)^* \cong M_\cs(I, V^*)$ is not isomorphic to $(V^*)^I = C(I, V^*)$, the space of $V^*$-valued arrays. If $I$ is finite, then both $(V^I)^*$ and $(V^*)^*$ are isomorphic to the product space $\S^{|I|}$. The (roman) subscript $\cs$ should not be confused with a kernel function $c$.} Evaluation of co-arrays is defined using the Lebesgue integral: $\varphi[a] := \int_I \varphi(\d i) a(i)$ for all $(\varphi, a) \in (V^I)^* \times V^I$.\footnote{We follow the physicists' notation and write the $V^*$-valued measure to the left of the $V$-valued array. For more details on $V^*$-valued measures, see Dunford and Schwartz \cite{dunford1957linear}.}
		
		Of particular importance is the scalar-valued case, where $V = \S = V^*$. Here, array space $\S^I = C(I,\S)$ consists of all (scalar-valued) continuous functions on $I$, and co-array space $(\S^I)^* = M_\cs(I,\S)$ consists of (scalar-valued) compactly supported Radon measures on $I$. When $I$ is finite, both array space $\S^I$ and co-array space $(\S^I)^*$ are isomorphic to $\S^{|I|}$. When $I$ is infinite, $\S^I$ and $(\S^I)^*$ are not isomorphic.

				

			
		\subsection{Random Arrays} \label{sect_randarrays}
		
		Let $\P$ be a probability measure on $V^I = C(I,V)$, representing the distribution of a random associative array, and let $\E$ denote its expectation operator. Co-arrays are continuous linear functionals, hence measurable functions $V \to \S$. The mean of an arbitrary co-array $\varphi \in (V^I)^*$ is the scalar quantity $\E[\varphi] = \int_{V^I} \varphi[a] \, \P(\d a) = \int_{V^I} \int_I \varphi(\d i) a(i) \, \P(\d a)$. The covariance is a (pre-)inner product on co-array space, and is defined by 
			\begin{eqnarray}
				\cov[\psi|\varphi] = \E[\bar{\psi} \varphi] - \E[\bar{\psi}] \E[\varphi] &=& \int_{V^I} \bar{\psi[a]} \varphi[a]  \, \P(\d a) - \E[\bar \psi] \E[\varphi] \label{def_covinnerprod1} \\
				&=& \int_{V^I} \left( \int_I \bar \psi(\d i)  \bar{a(i)}  \right) \left( \int_I \varphi(\d i')  a(i') \right) \, \P(\d a) - \E[\bar \psi] \E[\varphi] \label{def_covinnerprod2}
			\end{eqnarray}
		for all co-arrays $\psi, \varphi$.\footnote{Complex conjugation may be ignored when working with real scalars.} We assume throughout that $\P$ satisfies the finite-variance assumption, which means that the covariance $\cov[\psi|\varphi] < \oo$ for all co-arrays.\footnote{Since $I$ and $V$ are both separable, the space $C(I,V)$ is readily seen to be separable \cite{reed1972methods}. Therefore the separable-support hypothesis is automatically satisfied.} The mean is represented by an array $m \in V^I$, which satisfies the identity $\E[\varphi] = \varphi[m] = \int_I \varphi(\d i) m(i)$ for all co-arrays $\varphi$.

		\subsection{Covariance Functions (Scalar-Valued Case)} \label{sect_covfuns}
		
		It is useful to first illustrate the covariance representation in the scalar-valued case. Suppose that $V = \S = V^*$, with array space $\S^I = C(I,\S)$ and co-array space $(\S^I)^* = M_\cs(I,\S)$. Let $\delta_i \in (\S^I)^*$ denote the Dirac point-mass, which assigns weight $1 \in \S$ to index point $i \in I$.\footnote{Formally, $\delta_i(B) := 1$ if $i \in B$ and $:=0$ if $i \notin B$.} As a co-array, the Dirac point-mass acts by evaluation: $\delta_i[a] = a(i)$. The Dirac $\delta$-function is the continuous embedding $\delta : I \hookrightarrow (\S^I)^*$, which sends a point $i$ to its Dirac point-mass $\delta_i$. 
		
		The finite-variance assumption assures that the covariance inner product is well-defined, and we define the covariance function $c(i,i') := \cov\big[ \delta_i \big| \delta_{i'} \big]$.\footnote{In practice, the covariance function is typically provided explicitly, and encodes knowledge about the data to be observed (e.g., qualitative features like smoothness, or quantitative features like similarity degrees).}

		Using a covariance function, we may explicitly construct the covariance structure. Let $k : (\S^I)^* \to V$ denote the integral operator with kernel $c$, defined by $(k\varphi)(i) := \int_I \varphi(\d i') c(i,i')$.\footnote{Theorem 3.ii of \cite{vakhania1978covariance} ensures that the operator $k$ is continuous.} Define the one-point kernel functions $c_i(i') := c(i,i')$, and observe that $c_i = k(\delta_i)$ for each $i$. The embedded Cameron-Martin space $U \subseteq C(I, \S)$ is the Hilbert-closure of the span of $\{c_i\}$; its existence is justified in this case by the Moore-Aronszajn theorem \cite{berlinet2004reproducing}, or in general by our Lemma \ref{lem_threespace}.
		
		The following commutative diagram summarizes the relevant structure in the scalar-valued case:
			\begin{equation} \label{dia_covrepscalar}
						\begin{matrix} \xymatrix{ & (\S^I)^* \ar@{->}[rr]^k \ar@{->}[rd]^k && \S^I \\ I \ar@{^(->}[ur]^{\delta} \ar@{-->}[rr]_{k \circ \delta} && U \ar@{^(->}[ur] & } \end{matrix} \end{equation}
		Note that the probability measure $\P$ was not necessary for this construction: only the covariance kernel $c$ (and its integral operator $k$).
		
		The \emph{kernel trick} in machine learning is represented by the map $k \circ \delta : I \to U$, which allows us to compute geometric quantities in Hilbert space using only the covariance function. An important geometric invariance is the covariance metric on $I$, which is a continuous function $t_c : I \times I \to [0,\oo)$, and is defined by $t_c(i,i') := c(i,i) - 2\Re c(i,i') + c(i',i')$.\footnote{Note: $t_c$ is actually a pseudo-metric, and $t_c(i,i') = 0$ if and only if data at $i$ and $i'$ are maximally correlated.}

			\begin{rem}[Almost-Sure Continuity] \label{rem_finent}
				We assumed explicitly that $\P$ is a probability measure on $C(I,\S)$, which corresponds to the statistical hypothesis that random arrays are continuous almost surely. Dudley \cite{dudley2010sample} discovered that a.s.-continuity corresponds to a certain ``integrated entropy'' condition being satisfied; modern proofs are due to Talagrand \cite{talagrand1996majorizing}.
				
				Let $D$ be an arbitrary compact subset of $I$. For any $\epsilon > 0$, let $N_D(\epsilon)$ be the minimum number of $\epsilon$-balls needed to cover $D$ using the covariance metric $t_c$. Since $D$ is compact, $N_D(\epsilon)$ is finite for all $\epsilon$. Define the covariance entropy function $\Ent_D(\epsilon) := \log N_D(\epsilon)$, which is a measure of statistical complexity of data over the index set $D$. Define the integrated entropy $\IntEnt_D := \int_0^\oo \Ent_D(\epsilon) \, \d \epsilon$. Theorem 1.3.3 of \cite{adler2009random} ensures that a random array is almost-surely continuous if and only if $\IntEnt_D < \oo$ for all compact $D \subseteq I$.\footnote{Note: if $I$ is not compact, then it is possible for $\IntEnt_c^I = \oo$ but $\IntEnt_D < \oo$ for all compact $D \subseteq I$.}
				
				If $\IntEnt_D = \oo$, then a random array is almost surely discontinuous on $D$. In this case, no probability measure with covariance $c$ may be supported on the set $C(I, \S)$, and a different function space should be chosen.
			\end{rem}
			
		\subsection{Covariance Tensors (Vector-Valued Case)} \label{sect_covtens}
			
		The content of the next theorem is that we may represent the covariance kernel as a \emph{covariance tensor} $c : I \times I \to V \otimes_\cov V$. In infinite dimensions, tensor products are not unique, so this construction requires more care to define.\footnote{Covariance tensors generalize the scalar-valued setting, since $\S \otimes \S \cong \S$.} For an application of covariance tensors in random differential geometry, see \cite{lagatta2014geodesics}.\footnote{In \cite{lagatta2014geodesics}, index space is $I = \R^d$, and value space is $\operatorname{Sym}_d$, the space of $d \times d$ symmetric matrices. Array space $C(\R^d, \operatorname{Sym}_d)$ is used to represent random tensor fields on $\R^d$, when are then used to construct random Riemannian metrics. A similar construction is available for a general $d$-dimensional smooth manifold $I$, which has applications for kriging in general relativity.}

			\begin{rem}[Background:~Tensor Products]
				Recall that a tensor product $V \otimes V$ consists of all linear combinations of elements of the form $v \otimes v'$, closed under a suitable topology.\footnote{Addition of tensors is straightforward. Scalar multiplication of tensor products satisfies $s(v \otimes v') = (sv) \otimes v' = v \otimes (sv')$, contrary to the cartesian product, since $s(v,v') = (sv,sv')$.} When $V$ is finite dimensional, there is a unique tensor product, and $V \otimes V \cong V^{(\dim V)^2}$. When $V$ is infinite dimensional, it admits multiple types of tensor products, ranging from the projective tensor product $V \otimes_{\operatorname{proj}} V$ and the injective tensor product $V \otimes_{\operatorname{inj}} V$ \cite{ryan2002introduction}. A space is nuclear if and only if it has a unique tensor product.\footnote{Nuclearity of $V$ corresponds to the map $V \otimes_{\operatorname{proj}} V \to V \otimes_{\operatorname{inj}} V$ being an isomorphism. All finite-dimensional spaces are nuclear. There are no infinite-dimensional nuclear Banach spaces, though there exist nuclear Fr\'echet spaces.} If $H$ is a Hilbert space, there is a tensor product $H \otimes_{\Hilb} H$ which is itself a Hilbert space.\footnote{Infinite-dimensional Hilbert space is not nuclear.}
			\end{rem}
			
		We start by mimicking the scalar-valued construction in Section \ref{sect_covfuns}. First, define the extended Dirac point-mass $e\dash\delta_i$, which assigns weight $e \in V^*$ to index point $i \in I$.\footnote{Each $e\dash\delta_i$ is a $V^*$-valued measure. Formally, $e\dash\delta_{i}(B) := e$ if $i \in B$, and $:= 0$ if $i \notin B$. Note that $0\dash\delta_i = 0$ for all $i$.} Parametrization by $e$ is necessary since there may not exist a canonical unit $1 \in V^*$ (contrary to the scalar case).\footnote{When there is a unit (e.g., if $V^*$ is an algebra), then $1\dash\delta_i = \delta_i$ for all $i$.} As a co-array, the extended point-mass $e\dash\delta_{i}$ acts on an array $a$ by evaluation: $e\dash\delta_{i}[a] := e\big[ a(i) \big]$. In this setting, the Dirac $\delta$-function is a continuous embedding $\delta : V^* \times I \hookrightarrow (V^I)^*$.\footnote{By separability, the span of Dirac point-masses is dense in co-array space.} The covariance tensor is encoded by the scalar-valued function $c(e,i; e', i') := \cov\big[ e\dash\delta_i \big| e'\dash\delta_{i'} \big]$.\footnote{The covariance satisfies the scaling $c(se,i; e',i') = \bar s\cdot c(e,i; e',i') = c(e,i; \bar se',i')$ for all $s \in \S$.}
		
		Theorem \ref{thm_covrep} demonstrates that we may alternatively represent the covariance kernel as a covariance tensor $c : I \times I \to V \otimes_\cov V$, for a suitably chosen tensor product $V \otimes_\cov V$. The covariance operator $k : (V^I)^* \to V^I$ is defined by integrating against the tensor-valued covariance kernel: $(k \varphi)(i) := \int_I \varphi(\d i') c(i,i')$, where the action of $\varphi(\d i')$ is on the second tensor component of $c(i,i')$. The embedded Cameron-Martin space $U$ is the Hilbert closure of the one-point kernel functions $c_{e'i'}(i) := \big(k(e'\dash\delta_{i'})\big)(i) = e'\big[ c(i,i') \big]$, where the evaluation by $e'$ is on the second tensor component.\footnote{In the special case that $c(i,i') = \sum_n \alpha_n v_n(i) \otimes v'_n(i')$, then $(k\varphi)(i) = \int_I \varphi(\d i') c(i,i') = \sum_n \left( \alpha_n \int_I \varphi(\d i') v'_n(i) \right) v_n(i)$. In that case, the one-point functions satisfy $c_{e'i'}(i) = \sum_n \alpha_n v_n(i) \otimes e'\big[ v'_n(i') \big]$.} As in the scalar-valued setting, the space $U$ is the reproducing kernel Hilbert space (RKHS) for the kernel $c$.
				
		\begin{thm}[Covariance Kernel Representation Theorem] \label{thm_covrep}
				Let $\P$ be a finite-variance probability measure on $V^I = C(I,V)$. There exists a tensor product $V \otimes_\cov V$ and a (continuous) function $c : I \times I \to V \otimes_\cov V$ which is symmetric, non-negative-definite, and satisfies the relation $\cov[e\dash\delta_i | e'\dash\delta_{i'}] = (e \otimes_\cov e')\big[ c(i,i') \big]$ for all $e,e' \in V^*$ and $i,i' \in I$.
					
			\end{thm}
			
			The proof of Theorem \ref{thm_covrep} can be found in Appendix \ref{thm_covrepproof}. \newline

				The following commutative diagram summarizes the relevant structure in the vector-valued case:
					\begin{equation} \label{dia_covrep}
						\begin{matrix} \xymatrix{ & (V^I)^* \ar@{->}[rr]^k \ar@{->}[rd]^k && V^I \\ V^* \times I \ar@{^(->}[ur]^{\delta} \ar@{-->}[rr]_{k \circ \delta} && U \ar@{^(->}[ur] & } \end{matrix} \end{equation}

		The map $k \circ \delta : V^* \times I \to U$ again enables the kernel trick: geometric quantities in Hilbert space $U$ can be computed using either the scalar-valued kernel $c(e,i; e',i')$ or the tensor-valued kernel $c(i,i')$. 


			\begin{rem}[Almost-Sure Continuity]
				In the vector-valued case, almost-sure continuity can still be studied using covariance metrics, though now the construction is more subtle. Using the kernel trick, the covariance metric on $U$ pulls back to a covariance (pseudo-)metric, defined by $t_c(e,i; e',i') := c(e,i; e,i) - 2 \Re c(e,i; e',i') + c(e',i'; e',i')$. Let $\f : I \to V^*$ denote an arbitrary $V^*$-valued array (note: $\f$ is not a co-array), and define the pullback metric $t_c^\f(i,i') := t_c\big( \f(i), i; \f(i'), i' \big)$.\footnote{Under the covariance scaling, we have that $t_c(re, i; e', i') = r \cdot t_c(e,i; e', i') = t_c(e,i; r e', i')$ for all \emph{real} scalars $r \in \Reals \subseteq \S$. For the pulled-back metrics, the scaling satisfies $t_c^{r \f}(i,i') = t_c\big( r \f(i), i; r \f(i'), i' \big) = r^2 \cdot t_c^\f(i,i')$.} For any compact $D \subseteq I$, let $N_D^\f(\epsilon)$ denote the minimum number of $\epsilon$-balls needed to cover $D$ under the metric $t_c^\f$. Define the entropy function $\Ent_D^\f(\epsilon) := \log N_D^\f(\epsilon)$, and Talagrand's integrated entropy $\IntEnt_D^\f := \int_0^\oo \Ent_D^\f(\epsilon) \, \d \epsilon$. If $\IntEnt_D^\f < \oo$ for all compact $D \subseteq I$ \emph{and} all continuous $\f : I \to V^*$, then random $V$-valued arrays are continuous with probability one.
			\end{rem}
			
		\subsection{OLS Estimator for Array Transforms} \label{sect_transforming}
		
		We now formulate a general theorem about transforming data arrays in bulk. Using the OLS estimator, we may predict the source array given the output of the transformation. Suppose that $D$ and $I$ are any index spaces, and $V$ and $W$ are any value spaces. Let $i : D \to I$ be a continuous map, and let $w : V \to W$ be a continuous linear map. Consider the array spaces $V^I = C(I,V)$ and $W^D = C(D,W)$. There is a natural map $\Mech : V^I \to W^D$ defined by $\Mech(a)(d) := w\big( a\big( i(d) \big) \big)$.\footnote{The function $\Mech$ can be computationally implemented using the higher-order ``map'' function, applying the function $w$ index-wise to the array $a \circ i \in V^D$. In Clojure code, \texttt{(defn Upsilon [a] (map w (comp a i)))}.} Since $\Mech$ is continuous and linear, the structure theory of Section \ref{sect_structmaps} applies.\footnote{For the kriging example of Corollary \ref{cor_kriging}, $i : D \hookrightarrow I$ is the subset inclusion map, and $w : V \to V$ is the identity transformation. The restriction map $\Mech : V^I \to V^D$ is defined by $\Mech(a)(d) := a(d)$.}
		
		Assume that $\P$ is a probability measure on $V^I$, describing the law of a random array, and let $c : I \times I \to V \otimes_\cov V$ be its covariance kernel (as defined in Section \ref{sect_covtens}). Let $\Mech^* : (W^D)^* \to (V^I)^*$ denote the adjoint operator, which acts on co-arrays.\footnote{The adjoint $\Mech^*$ is defined by sending a co-array $\nu \in (W^D)^*$ to a certain co-array $\Mech^*\nu \in (V^I)^*$, which satisfies the identity $\int_I (\Mech^* \nu)(\d i) a(i) = \int_D \nu(\d d) (\Mech a)(d) = int_D \nu(\d d) w\big( a\big( i(d) \big) \big)$ for all arrays $a \in V^I$.} Let $k : (V^I)^* \to V^I$ be the covariance operator (i.e., the integral operator of $c$). Let $\P_\Mech := \P \circ \Mech^{-1}$ denote the push-forward measure on $W^D$, describing the outcome of a random sample of $\Mech$.\footnote{i.e., if $a$ is a random $V$-valued array with law $\P$, then $\P_\Mech$ is the law of the random $W$-valued array $\Mech(a)$.} It is straightforward to see that the covariance tensor of $\P_\Mech$ is a continuous function $c_D : D \times D \to W \otimes_\cov W$, defined by $c_D(d,d') := \big( w \otimes_\cov w \big)\big( c\big( i(d), i(d') \big) \big)$.\footnote{If the covariance admits the representation $c(i,i') = \sum \alpha_n v_n(i) \otimes v_n'(i')$, then $c_D(d,d') = \sum_n \alpha_n w\big( v_n\big(i(d)\big) \big) \otimes w\big( v_n' \big(i(d')\big) \big)$.} 
		
		Let $U$ be the Cameron-Martin space of $k$, which is generated by one-point functions $c_{ei} := e \big[ c(i,\cdot) \big]$, where the evaluation of $e \in V^*$ is on the first tensor component. Let $U_D$ be the Cameron-Martin space of $k_D$, which is generated by one-point functions $c_{D,fd} := f \big[ c_D(d, \cdot) \big] := f \big[ (w \otimes_\cov w)\big( c\big( i(d), i(\cdot)\big) \big) \big]$, where the evaluation of $f \in W^*$ is on the first tensor component. 		
		
		By the Main Theorem (Theorem \ref{thm_iso}), the restricted data mapping $\Mech : U \to U_D$ corresponds to orthogonal projection onto the subspace $\hat{U_D}$.\footnote{The space $\hat{U_D}$ is generated by the one-point kernel functions $\hat c_{ed} := e \big[ c\big( i(d), \cdot \big) \big]$, where the evaluation of $e \in V^*$ is on the first tensor component and $d$ ranges over $D$.} Consequently, the map $\Mech : \hat U_D \to U_D$ is an isomorphism of Hilbert spaces, and therefore the inverse map $\Mech^{-1} : U_D \to \hat U_D$ is well-defined and continuous. Assuming that $c$ is strictly positive-definite, the OLS estimator $\hat \Mech_\OLS : W^D \to V^I$ is well-defined and continuous on the maximum possible domain.\footnote{For a parametrized model $\theta \mapsto \P_\theta$, if $W$ is a Banach space and $D$ is compact, then Lemma \ref{lem_ctyshort} ensures that $(\theta,y) \mapsto \hat \Mech_{\OLS,\theta}(y)$ is jointly continuous.}
		
			\begin{thm}[OLS Estimator for General Array Transforms] \label{thm_OLSarrays}
				Suppose that $c$ is strictly positive-definite, and suppose that $w : V \to W$ has dense image. The Ordinary Least Squares estimator $\hat \Mech : W^D \to V^I$ is well-defined and continuous on the maximal possible domain $W^D$.\footnote{If $c$ is not strictly positive-definite or $w$ is not dense, then the domain of the OLS estimator is the closed subspace $\bar{U_D} \subseteq W^D$.} The OLS estimator is a ``best linear unbiased estimator'' (BLUE). 
			\end{thm}
			\begin{proof}
				By Theorem \ref{thm_OLS}, the OLS estimator is a well-defined, continuous linear operator. Since $c$ is strictly positive-definite, the Cameron-Martin space $U$ is dense in $V^I$. Since $w$ is dense, $c_D$ is strictly positive definite, hence $U_D$ is dense in $W^D$. Thus the maximal domain of the OLS estimator is $\bar{U_D} = W^D$. The OLS estimator is a BLUE by virtue of the Gauss-Markov theorem (Theorem \ref{thm_GMT}).
			\end{proof}
		
		Define the (extended) Dirac $\delta$-functions $\delta : V^* \times I \to (V^I)^*$ and $\delta_D : W^* \times D \to (W^D)^*$ as in Section \ref{sect_covtens}. Let $w^* \times i : W^* \times D \to V^* \times I$ denotes the cartesian product map.\footnote{i.e., $(w^* \times i)(f, d) := \big( w^* f, i(d) \big)$ for all $f \in W^*$ and $d \in D$.} The following commutative diagram encodes the full structure in this setting:
			\begin{equation} \label{dia_multirep}
					\begin{matrix} \xymatrix{ & (W^D)^* \ar@{->}[rr]^{\Mech^*} && (V^I)^* \ar@{->}[rr]^k \ar@{->}[rd]^k && V^I \ar@{->}@/^/[rr]^{\Mech} && W^D \ar@{_(-->}@/^/[ll]^{\hat \Mech_\OLS} \\ W^* \times D \ar@{^(->}[ur]^{\delta_D} \ar@{->}[rr]_{w^* \times i} && V^* \times I \ar@{^(->}[ur]^{\delta} \ar@{-->}[rr]_{k \circ \delta} && U \ar@{^(->}[ur] \ar@{->>}@/^/[rr]^{\Mech} && U_{D} \ar@{_(->}@/^/[ll]^{\Mech^{-1}} \ar@{^(->}[ur]  } \end{matrix} \end{equation}

		\section{The Problem of Classification and Support-Vector Machines} \label{sect_classification}
		
		Let $D_0, D_1$ be two disjoint (compact) subsets of index space $I$, representing observed data of two different label types. The problem of \emph{classification} is to find a suitable partition $L_0, L_1$ of $I$ given the training data $D_0, D_1$; a \emph{classifier} is the indicator function $\kappa(i) := 1_{L_1}(i)$.\footnote{i.e., $\kappa(i) = 1$ if $i \in L_1$, and $=0$ if $i \notin L_1$. The partition must be consistent with the data: $D_0 \subseteq L_0$, $D_1 \subseteq L_1$, $L_0 \cup L_1 = I$ and $L_0 \cap L_1 = \varnothing$.} Given a kernel function $c$, Cortes and Vapnik \cite{cortes1995support} designed the support-vector machines (SVM) algorithm to solve the classification problem on a finite index space $I$. 
		
		More generally, Hein, Bousquet and Sch\"olkopf \cite{hein2005maximal} present the SVM algorithm for a metric index space (and finite observation sets $D_0, D_1$). Their Theorem 5 demonstrates that the SVM classifier for $c$ is the maximum-margin classifier for the covariance metric $t_c(i,i') := c(i,i) - 2 \Re c(i,i') + c(i',i')$, and is therefore optimal. In Section \ref{sect_SVM}, we present a general version of the SVM classifier for arbitrary index spaces (allowing the kernel to vary) and compact observation sets $D_0, D_1$.\footnote{It is possible to extend the classifier to non-compact observation sets, but we do not do so here.}
		
			\begin{rem}
				The case of multiple labeled data sets $D_0, D_1, D_2, \cdots$ is typically solved by applying multiple binary SVM classifications. This case can also be solved by considering a Voronoi tessellation of $I$ with respect to the covariance metric $m_c$.
			\end{rem}
		

		
		\subsection{The SVM Algorithm} \label{sect_SVM}
		
		We recall the formal setup from Section \ref{sect_covfuns}. Consider (scalar-valued) array space $\S^I := C(I, \S)$ and co-array space $(S^I)^* := M_\cs(I,\S)$. Let $c$ be a symmetric, non-negative-definite kernel on $I$, meaning that $c(i,i') = \bar{c(i',i)}$ for all $i,i' \in I$ and \eqref{def_nonneg} is satisfied. Let $k : (\S^I)^* \to \S^I$ be the integral operator with kernel $c$.\footnote{i.e., $(k\varphi)(i) = \int_I c(i,i') \, \varphi(\d i')$ for any co-array $\varphi \in (\S^I)^*$. Co-array space consists of compactly supported Radon measures on $I$.} The operator $k$ is known in the machine-learning literature as the \emph{mean map} \cite{muandet2012learning}. Let $U \subseteq \S^I$ be the reproducing kernel Hilbert space (RKHS) for $c$, i.e., the Hilbert closure of the one-point kernel functions $c_i := c(i,\cdot)$.\footnote{The dimension $\dim U + 1$ is called the Vapnik-Chervonenkis dimension of the support-vector machine, and represents the complexity of data classifiable.} The kernel trick $k \circ \delta : I \hookrightarrow U$ sends an index point $i$ to its one-point kernel function $c_i$, and is a continuous embedding of index space into Hilbert space.\footnote{In fact, $I$ is a metric space when equipped with the covariance metric, and the kernel trick is an isometric embedding of $I$ into $U$. By Proposition 4 of \cite{hein2005maximal}, only the metric $m_c$ is relevant for the SVM classifier, not the kernel $c$.}
		
		Since $D_0$ and $D_1$ are compact, and the kernel trick is continuous, the images $(k \circ \delta)D_0$ and $(k \circ \delta)D_1$ are compact subsets of Hilbert space. Define their respective closed convex hulls by $\hat D_0$ and $\hat D_1$, which are themselves compact.\footnote{The closed convex hull of a compact set is compact \cite[Theorem 5.35]{aliprantis2006infinite}.} 
		
		Suppose that the kernel is strictly positive-definite, so that the sets $\hat D_0$ and $\hat D_1$ are guaranteed to be linearly separable in $U$. Define their maximum-margin separation distance by $\rho := \| \hat D_0 - \hat D_1 \| := \min \| u_0 - u_1 \| > 0$, where the minimum is over pairs $(u_0,u_1) \in \hat D_0 \times \hat D_1$. Let $\Sigma := \big\{ (u_0, u_1) : \|u_0 - u_1 \| = \rho \big\}$ denote the set of solution pairs, which is non-empty by compactness. Define the two ``solution facets'' $\Sigma_0 \subseteq \hat D_0$ and $\Sigma_1 \subseteq \hat D_1$ by projecting $\Sigma$ onto its marginal components.

			\begin{lem}[Fundamental Separation Lemma] \label{lem_funsep}
			Suppose that $c$ is strictly positive-definite.
				\begin{enumerate}
				\item \label{lem_funsep1} The solution set $\Sigma$ is convex and non-empty.
				\item \label{lem_funsep2} There exists a unique vector $\w \in U$ (called the \emph{separation vector}) so that $u_1 = u_0 + \w$ for all solution pairs $(u_0, u_1) \in \Sigma$.\footnote{Consequently, $\Sigma \cong \Sigma_0 \cong \Sigma_1$, and $\Sigma_1 = \Sigma_0 + \w$.} The vector $\w$ has length $\rho$ in $U$.
				
				\item \label{lem_funsep3} There exist \emph{maximum-marginal hyperplanes} $M_0, M_1 \subseteq U$ which are orthogonal to the solution vector $\w$, are separated by Hilbert distance $\rho$, and separate $\hat D_0$ and $\hat D_1$.\footnote{i.e., for all $(d_0, d_1) \in \hat D_0 \times \hat D_1$ and $(m_0, m_1) \in M_0 \times M_1$, we have $\proj_\w(d_0 - m_0) \le 0 \le \proj_\w(d_1 - m_1)$. Since $\proj_\w(m_1 - m_0) = \rho$, it follows that $\proj_\w(d_1 - d_0) \ge \rho$. \label{foot_sep}} Each solution facet $\Sigma_\ell$ is a convex subset of the convex set $M_\ell \cap \hat D_\ell$, and $\Sigma \subseteq M_0 \times M_1$.\footnote{The solution facet $\Sigma_\ell$ is the projection of $\Sigma$ onto the $\ell$th component.}
					
				\item \label{lem_funsep4} The \emph{support-vector sets} $X_\ell := M_\ell \cap (k \circ \delta)D_\ell \subseteq U$ are non-empty and compact, and every solution vector $u_\ell \in \Sigma_\ell$ can be written as a convex combination of support vectors from $X_\ell$.\footnote{That is, for each solution $u_\ell$, there exists a probability measure $\p_\ell$ supported on $X_\ell$ such that $u_\ell = \int_{X_\ell} x_\ell \, \p_\ell(\d x_\ell).$ This is a Pettis integral, and is characterized by the fact that $\langle v, u_\ell \rangle = \int_{X_\ell} \langle v, x_\ell \rangle \, \p_\ell (\d x_\ell)$. } Every support vector $x_\ell = c(d_\ell,\cdot)$ is the one-point kernel function for some ``support point'' $d_\ell \in D_\ell$.
				
				\item \label{lem_funsep5} The separation vector $\w$ can be written as a \emph{unique} difference of convex combinations of support vectors. That is, there exist probability measures $\nu_0$ and $\nu_1$ supported on $X_0$ and $X_1$, respectively, so that
					\begin{equation} \label{eqn_wdec1}
						\w = \int_{X_1} x_1 \, \nu_1(\d x_1) - \int_{X_0} x_0 \, \nu_0(\d x_0), \end{equation}
				where the integrals are Pettis integrals.\footnote{Decomposition \eqref{eqn_wdec1} is characterized by the fact that $\langle v, \w \rangle = \int_{x_1} \langle v, x_1 \rangle \, \nu_1(\d x_1) - \int_{x_0} \langle v, x_0\rangle \, \nu_0(\d x_0)$ for any $v \in U$.} Furthermore, there exist probability measures $\tilde \nu_0$ and $\tilde \nu_1$ on the support-point sets $\tilde X_0 := (k \circ \delta)^{-1} X_0$ and $\tilde X_1 := (k \circ \delta)^{-1} X_1$, respectively, so that the array $\w \in U$ satisfies
					\begin{equation} \label{eqn_wdec2}
						\w(i) = \langle \delta_i, \w \rangle = \int_{\tilde X_1} c(i, d_1) \, \tilde \nu_1(\d d_1) - \int_{\tilde X_0} c(i, d_0) \, \tilde \nu_0(\d d_0) \end{equation}
					for all $i$. Therefore, if the combination measures $\tilde \nu_\ell$ are known, the solution array $i \mapsto \w(i)$ can be computed using solely the covariance kernel.
				\end{enumerate}
			\end{lem}
			
		The proof of Lemma \ref{lem_funsep} can be found in Appendix \ref{lem_funsepproof}. \newline

		 
		 
		 With the separation vector $\w$ in hand, it is easy to construct the SVM classifier. Let $x_0 \in X_0$ denote an arbitrary support vector from the initial data set. We partition Hilbert space $U$ into two parts by projecting any point onto the ray from $x_0$ in the direction of $\w$:
		 	\begin{equation} \label{def_hilbpart}
				U_0 := \big\{ u : \proj_\w(u - x_0) < \tfrac \rho 2 \big\} \qquad \mathrm{and} \qquad U_1 := \big\{ u : \proj_\w(u - x_0) \ge \tfrac \rho 2 \big\}, \end{equation}
		where we break a tie in favor of the second label. We construct the partition of index space by pulling the sets $U_0$ and $U_1$ back through via kernel map: $L_0 := (k \circ \delta)^{-1} U_0$ and $L_1 := (k \circ \delta)^{-1} U_1$.\footnote{By construction, $L_0$ is open and $L_1$ is closed.} The SVM classifier $\kappa(i) := 1_{L_1}(i)$ is the indicator function of $L_1$. 
		
		This reduces the problem of classification to finding the convex coefficients $\tilde \nu_0, \tilde \nu_1$ as in \eqref{eqn_wdec2}. In the finite case, the convex coefficients can be found by solving the Karush-Kuhn-Tucker problem, as observed by Cortes and Vapnik \cite[Appendix A]{cortes1995support}. When the data sets $D_0, D_1$ are infinite, the problem is more difficult, and techniques from infinite-dimensional optimal control are necessary to compute the coefficients $\tilde \nu_0, \tilde \nu_1$. The appropriate generalization of the KKT conditions in the infinite-dimensional setting is the Pontryagin maximal principle \cite[Theorem 6.1]{fattorini1991necessary}; recent results can be found in \cite{krastanov2011pontryagin}.
		
		Let $\kappa_{D_0,D_1,c}$ denote the SVM classifier for observation data sets $D_0$ and $D_1$ with kernel $c$. Let $\D_c$ denote the space of compact subsets of $I$, equipped with the Hausdorff topology induced by metric $t_c$.\footnote{Let $D^\epsilon := \{ i : t_c(i,d) < \epsilon \mbox{ for some $d \in D$} \}$ denote the $\epsilon$-ball around any $D \in \D_c$. The Hausdorff distance $t_{c,\operatorname{Haus}}(D_1, D_2)$ is the minimum $\epsilon$ so that $D_1 \subseteq D_2^\epsilon$ and $D_2 \subseteq D_1^\epsilon$.} Observe that, as subsets of function space $\S^I$, the compact convex sets $\hat D_0$ and $\hat D_1$ vary jointly continuously in $(D_0, D_1,c)$. Consequently, the classifier $(D_0, D_1,c) \mapsto \kappa_{D_0, D_1, c}$ varies jointly continuously, which ensures that SVM is a good solution to the problem of classification.
		
			\begin{rem}[OLS as a Fuzzy Classifier]
				Using the OLS kriging predictor, we construct a fuzzy classifier $\lambda : I \to \S$.\footnote{A fuzzy classifier takes value $1$ on $D_1$, $0$ on $D_0$ and continuously interpolates in between.} Define the unified observation space $D := D_0 \cup D_1$, and define the data $y(d) := 1$ if $d \in D_1$ and $:= 0$ if $d \in D_0$.\footnote{The function $y : D \to \{0,1\}$ is continuous since the sets $D_0$ and $D_1$ are closed and disjoint.} Consider array space $\S^D := C(D,\S)$, and define the OLS kriging predictor $\hat \Mech_\OLS : \S^D \to \S^I$ as in Section \ref{sect_kriging}. We define the fuzzy OLS classifier by $\lambda(i) := \hat \Mech_\OLS(y)(i)$, i.e., simply by extending the data using OLS estimation. 
				
				The Gauss-Markov theorem implies that the fuzzy classifier $\lambda$ is geometrically optimal, minimizing estimated variance and mean-squared error. By \cite[Theorem 5]{hein2005maximal} the SVM classifier $\kappa$ is also geometrically optimal, being a maximum-margin classifier for the covariance metric $t_c$. This suggests a deeper connection between OLS and SVM, which we leave open for future work. 
			\end{rem}

	\section{Conclusion}
		
		In this article, we have presented a structural approach for resolving problems in non-parametric statistics. Our general formalism allows us to deal simultaneously with finite-dimensional and infinite-dimensional spaces. A minimality of assumptions ensures maximal applicability of our solutions. Our Main Theorem is a novel result in the structure theory of probability measures on topological vector spaces, and enables the construction of an Ordinary Least Squares estimator in absolute generality. We use the OLS estimator to solve the problem of linear regression, and we construct a stochastic OLS estimator to solve the problem of linear conditioning for ``uncorrelated implies independent'' measures. We prove a new representation theorem for covariance tensors, show that OLS can be used to define a good kriging predictor for vector-valued arrays on arbitrary index spaces, and we construct a general version of a support-vector machines classifier. We hope that our article has illuminated some of the deeper connections between probability theory, statistics and machine learning, and that future researchers can learn from and extend our results.

		\part*{Appendices}
		
		\appendix		
		
			

		


			\section{Strong Continuity Lemma} \label{app_cty}

			Define the extended hyperparameter space $\Theta_f := \big\{ (\theta, y) : \supp \P_{\Mech,\theta} \big\}$.\footnote{The space $\Theta_f$ is equipped with the subspace topology inherited from $\Theta \times Y$.} Define the OLS estimator as in Section \ref{sect_OLS}. Define the joint OLS estimator $\OLSopnone : \Theta_f \to \Para$ by $\OLSoptheta(\theta,y) := \OLSoptheta(y)$. The following is the long version of Lemma \ref{lem_ctyshort}.

				\begin{lem}[Strong Continuity Lemma, long version] \label{lem_cty}
						Let $\|\cdot\|_\Para$ be a continuous seminorm on $\Para$, and let $\|\cdot\|_Y$ be a continuous norm on $Y$ (for example, if $\Para$ is a Fr\'echet space and $Y$ is a Banach space). Define the operator norm of the OLS estimator (relative to $\|\cdot\|_\Para$ and $\|\cdot\|_Y$) by
							\begin{equation} \label{lem_ctyuniform}
								M_{\Mech,\theta} := \| \OLSoptheta\|_\op := \sup_{y \in \bar{A_{\Mech,\theta}}} \frac{\|\OLSoptheta(y) - m_\theta \|_\Para}{\|y - m_{\Mech,\theta} \|_Y} = \sup_{u \in U_{\Mech,\theta}} \frac{\|\OLSoptheta(u) \|_\Para}{\|u\|_Y} = \sup_{e \in Y^*} \frac{\| k_\theta \Mech^* e \|_\Para}{\|\Mech k_\theta \Mech^* e\|_Y}. \end{equation}
							 Then $M_{\Mech,\theta} < \oo$, and the real-valued function $\theta \mapsto M_{\Mech,\theta}$ is continuous.\footnote{The operator norm does not depend on the size of the means $\|m_\theta\|_\Para$ and $\|m_{\Mech,\theta}\|_Y$. It is important that $\|\cdot\|_Y$ be a norm (and not just a seminorm) in order to avoid dividing by zero.} Consequently, the joint OLS estimator $(\theta,y) \mapsto \OLSoptheta(y)$ is jointly continuous.
							 
						Similarly, the operator norm of the difference $\OLSoppara{\theta'} - \OLSoptheta$ satisfies
							\begin{equation}
								\Delta_{\Mech,\theta,\theta'} := \| \OLSoppara{\theta'} - \OLSoptheta\|_\op = \sup_{e \in Y^*} \frac{\| (k_{\theta'} - k_\theta) \Mech^* e \|_\Para}{\|\Mech (k_{\theta'} - k_\theta) \Mech^* e \|_Y}. \end{equation}
						Furthermore, $\Delta_{\Mech,\theta,\theta'}  < \oo$, and the real-valued function $(\theta, \theta') \mapsto \Delta_{\Mech,\theta,\theta'}$ is continuous. Consequently, the joint difference function $(\theta,y; \theta', y') \mapsto \OLSoppara{\theta'}(y') - \OLSoptheta(y)$ varies jointly continuously.
				\end{lem}
				

			
			\begin{proof}
					
					The equality of the different quantities in \eqref{lem_ctyuniform} and \eqref{lem_ctyuniform} is clear, since $\Mech k_\theta \Mech^* Y^*$ and $U_{\Mech,\theta}$ are both dense in $\bar{U_{\Mech,\theta}} = \bar{\Mech k_\theta \Mech^* Y^*} = \bar{A_{\Mech,\theta}} - m_{\Mech,\theta}$. The quantity $M_{\Mech,\theta}$ is finite since the OLS estimator $\OLSoptheta$ is a continuous operator.\footnote{In a seminormed space, continuity implies boundedness.}
					
					For each $e \in Y^*$, define $M_{\Mech,\theta}(e) := \frac{\| k_\theta \Mech^* e \|_\Para}{\|\Mech k_\theta \Mech^* e\|_Y}$, so that $M_{\Mech,\theta} = \sup_{e \in Y^*} M_{\Mech,\theta}(e)$.\footnote{We set $M_{\Mech,\theta}(e) := 1$ if $\|\Mech k_\theta \Mech^* e\|_Y = 0$. Since $\|\cdot\|_Y$ is a norm, $\|\Mech k_\theta \Mech^* e\|_Y = 0$ implies that $e \in \null \Mech k_\theta \Mech^*$, which implies that $\Mech^* e \in \null k_\theta$ (since $\Mech$ is an isomorphism on $U_\theta$), which implies that $\|k_\theta \Mech^* e\|_\Para = 0$. If $\|\cdot\|_Y$ were instead a seminorm, it could be possible for $\|\Mech k_\theta \Mech^* e \|_Y = 0$ but $\|k_\theta \Mech^*\| \ne 0$, forcing $M_{\Mech,\theta} = \oo$.} By continuity of $\theta \mapsto k_\theta$, the function $(\theta,e) \mapsto M_{\Mech,\theta}(e)$ is jointly continuous in $\theta$ and $e$.\footnote{This claim also relies on the continuity of $\Mech$ and $\Mech^*$, the seminorm $\|\cdot\|_\Para$, and the norm $\|\cdot\|_Y$.} The sets $\{ e : M_{\Mech,\theta}(e) \ge M_{\Mech,\theta} - \epsilon \}$ are convex, which implies that $\theta \mapsto M_{\Mech,\theta}$ varies continuously. 
					
					The corresponding arguments for the difference function $(\theta, \theta') \mapsto \OLSoppara{\theta'} - \OLSoptheta$ and its operator norm $\Delta_{\Mech,\theta,\theta'}$ are similar, so we omit them. 
					
					To see that $(\theta,y) \mapsto \OLSoptheta(y)$ varies jointly continuously, let $(\theta^\gamma,y^\gamma) \to (\theta,y)$ be a convergent net in $\Theta_f$. By the triangle inequality, 
						\begin{eqnarray}
							\big\| \OLSoppara{\theta^\gamma}(y^\gamma) - \OLSoptheta(y) \big\|_\Para &\le& \big\| \OLSoppara{\theta^\gamma}(y^\gamma) - \OLSoptheta(y^\gamma) \big\|_\Para + \big\| \OLSoptheta(y^\gamma) - \OLSoptheta(y) \big\|_\Para \nonumber \\
							&\le& \big\| \OLSoppara{\theta^\gamma} - \OLSoptheta\big\|_\op \big\| y^\gamma \big\|_Y + \big\| \OLSoptheta\big\|_\op \big\| y^\gamma - y \big\|_Y \nonumber \\
							&=& \Delta_{\Mech,\theta,\theta^\gamma} \big\| y^\gamma \big\|_Y + M_{\Mech,\theta} \big\| y^\gamma - y \big\|_Y \to 0, 
						\end{eqnarray}
					since $\Delta_{\Mech,\theta,\theta^\gamma} \to \Delta_{\Mech,\theta,\theta} = 0$, $\big\| y^\gamma \big\|_Y \to \big\| y \big\|_Y$, $M_{\Mech,\theta}$ is constant, and $\big\| y^\gamma - y \big\|_Y \to 0$. This proves joint continuity. The proof of joint continuity for the difference function is similar, so we omit it. 					
				\end{proof}

			\section{Proofs of Technical Results}
			
			\subsection{Proof of Three-Space Lemma (Lemma \ref{lem_threespace})} \label{lem_threespaceproof}
			
				Let $\iota_\theta^* : \Para^* \twoheadrightarrow H_\theta$ be the linear inclusion map, which sends a functional $f$ to its equivalence class $\iota_\theta^*f := [f]_\theta \in H_\theta$.\footnote{The equivalence relation on $\Para^*$ is defined by $e \sim_\theta f$ if $\var_\theta(e - f) = 0$. Equivalently, $e = f$ with $\P_\theta$-probability one.} By construction, the map $\iota_\theta^*$ has dense image in $H_\theta$. 
				
				We define the continuous injective map $\iota_\theta : H_\theta \hookrightarrow \Para$ as follows. On the dense subspace $\iota_\theta^* \Para^* \subseteq H_\theta$, we define $\iota_\theta\!\big( \iota_\theta^* f \big) := k_\theta(f)$. Since the map $\iota_\theta^*$ is isometric (by construction), we have that 
					\begin{equation} \label{eqn_covdense}
						\cov_\theta\!\big[ \iota_\theta^* f \big| \iota_\theta^* e \big] = \cov_\theta\!\big[ f \big| e \big] = \bar f\big[ k_\theta e \big] = \bar f\big[ \iota_\theta\!\big( \iota_\theta^* e \big) \big]. \end{equation}
				
				Since $k$ is a continuous operator, the definition of $\iota_\theta$ extends to the entire Hilbert space $H_\theta$.\footnote{i.e., if $f^\gamma$ is a net in $\Para^*$ such that $\iota_\theta^* f^\gamma \to h$ in $H_\theta$, we define $\iota_\theta(h) := \lim_s k_\theta(f^\gamma)$. To see that this does not depend on the specific choice of net, let $\iota_\theta^* \tilde f^\gamma \to h$ be a second convergent net in $H_\theta$. For any $e \in \Para^*$, we use \eqref{eqn_covdense} to see that $e\big[ k_\theta\!\big( f^\gamma - \tilde f^\gamma \big) \big] = \cov_\theta\!\big[ \bar e \big| f^\gamma - \tilde f^\gamma \big] \to \cov_\theta[e | h - h] = 0$. Consequently, the diagram \eqref{dia_classic} commutes. Since the dual space separates points, this implies that the map $\iota_\theta$ is well-defined.} Consequently, the diagram \eqref{dia_classic} commutes. 
				
				Injectivity of $\iota_\theta$ follows, since the construction of $H_\theta$ involved quotienting out by the null space of $k$. Since $\iota_\theta^* \Para^*$ is dense in $H_\theta$, \eqref{eqn_covdense} implies that $\cov_\theta\!\big[ \iota_\theta^* f \big| h \big] = \bar f\big[ \iota_\theta h \big]$ for all $f \in \Para^*$ and $h \in H_\theta$.
				
				The separable-support hypothesis implies that $k_\theta \Para^*$ and $\bar{k_\theta \Para^*}$ are separable subspaces of $\Para$. The three-space diagram implies that the embedded Cameron-Martin space $U_\theta := \iota_\theta H_\theta$ is a dense subspace of $\bar{k_\theta \Para^*}$, and is therefore separable. The map $\iota_\theta$ is a continuous embedding, so the separable Hilbert space $U_\theta$ is isomorphic to the abstract Cameron-Martin space $H_\theta$ (hence the latter is also separable). Translation by $m_\theta$ is an isomorphism in $\Para$, so $U_\theta$ and $A_\theta$ are isomorphic as affine Hilbert spaces. \hfill \qed

						\subsection{Proof of Six-Space Lemma (Lemma \ref{lem_twospaces})} \label{lem_twospacesproof}
			


				First, we construct the inclusion map $\eta_{\Mech,\theta} : H_{\Mech,\theta} \hookrightarrow \hat H_{\Mech,\theta}$. On the dense subspace $\iota_{\Mech,\theta}^* Y^* \subseteq H_{\Mech,\theta}$, we define $\eta_{\Mech,\theta}\big(\iota_{\Mech,\theta}^* e\big) := \big( \iota_\theta^* \circ \Mech^* \big)(e)$. This map is continuous, so extends to the entire space $H_{\Mech,\theta}$. By construction, the inclusion  $\eta_{\Mech,\theta}$ is clearly an isomorphism from $H_{\Mech,\theta}$ onto $\hat H_{\Mech,\theta}$. 
				
				Since the space $\hat H_{\Mech,\theta}$ is a closed subspace of $H_{\Mech,\theta}$, the (internal) orthogonal projection map $\hatpi_{\Mech,\theta} : H_\theta \twoheadrightarrow \hat H_{\Mech,\theta}$ is guaranteed to exist. We define the (external) orthogonal projection map by $\pi_{\Mech,\theta} = \eta_{\Mech,\theta}^{-1} \circ \hatpi_{\Mech,\theta}$.
				
				We next prove the two identities \eqref{eqn_twoidentities}. On the domain $Y^*$, we have
					\begin{equation}
						\pi_{\Mech,\theta} \circ \iota_\theta^* \circ \Mech^* = \big( \eta_{\Mech,\theta}^{-1} \circ \hatpi_{\Mech,\theta} \big) \circ \iota_\theta^* \circ \Mech^* =  \eta_{\Mech,\theta}^{-1} \circ \iota_\theta^*  \circ \Mech^*  = \iota_{\Mech,\theta}^* \end{equation}
				since $\hatpi_{\Mech,\theta}$ is the identity map on $\iota_\theta^* Y^* \subseteq \hat H_{\Mech,\theta}$, and $\iota_\theta^*  \circ \Mech^* = \eta_{\Mech,\theta} \circ \iota_{\Mech,\theta}^*$.
				
				To prove the second identity, let $e \in Y^*$, and observe that
					\begin{equation}
						\big( \Mech \circ \iota_\theta \circ \eta_{\Mech,\theta} \big)\big( \iota_{\Mech,\theta}^* e \big) = \big( \Mech \circ \iota_\theta \circ \iota_\theta^* \circ \Mech^* \big)(e) = \Mech k_\theta \Mech^* e = \iota_{\Mech,\theta}\big( \iota_{\Mech,\theta}^* e \big). \end{equation}
				Since the space $\iota_{\Mech,\theta}^* Y^*$ is dense in $H_{\Mech,\theta}$ and all maps are continuous, this proves the second identity. The commutativity of the diagram follows.	\hfill \qed

			\subsection{Proof of Theorem \ref{thm_OLS}} \label{thm_OLSproof}
			
				Since $\OLSoptheta$ is defined to be the continuous extension of $\Mech^{-1} : A_{\Mech,\theta} \to \hat A_{\Mech,\theta}$, it is well-defined and continuous. To see that $\OLSoptheta$ is a right-inverse to $\Mech$, note that $(\Mech \circ \OLSoptheta)(\ell) = (\Mech \circ \Mech^{-1})(\ell)$ for all $\ell \in A_{\Mech,\theta}$. Consequently, the continuous map $\Mech \circ \OLSoptheta$ agrees with the identity operator on the dense subspace $L_{\Mech,\theta} \subseteq \bar{L_{\Mech,\theta}}$, hence $\Mech \circ \OLSoptheta$ equals the identity operator. This proves that $\OLSoptheta$ is a well-defined, continuous linear estimator.
				
				By the Main Theorem, the isomorphism $\Mech^{-1}$ maps the center point $m_{\Mech,\theta} \in A_{\Mech,\theta}$ to the center point $m_\theta \in \hat A_{\Mech,\theta}$, so $\OLSoptheta$ is unbiased, proving part \ref{thm_OLS2}. 
				

				

				Let $\Mech^2  : Y \to Z$ be a continuous linear map, and consider the continuous linear composition $\Mech^2 \circ \Mech^1 : \Para \to Z$. Let $A_{\Mech^2,\theta}$ and $A_{\Mech^2 \circ \Mech^1,\theta}$ denote the two affine Cameron-Martin spaces in $Z$ (which in general will not be equal), and let $\hat A_{\Mech^2,\theta} \subseteq Y$ and $\hat A_{\Mech^2 \circ \Mech^1,\theta} \subseteq \Para$ denote their lifts. By the Main Theorem, the inverses of the restriction maps $(\Mech^2)^{-1} : A_{\Mech^2,\theta} \to \hat A_{\Mech^2,\theta}$ and $(\Mech^2 \circ \Mech^1)^{-1} : A_{\Mech^2\circ \Mech^1,\theta} \to \hat A_{\Mech^2\circ \Mech^1,\theta}$ are isomorphisms of affine Hilbert spaces. Therefore, $(\Mech^2 \circ \Mech^1)^{-1} = (\Mech^1)^{-1} \circ (\Mech^2)^{-1}$ (by co-functoriality of the inverse), hence $\OLSoptheta^{2,1} = \OLSoptheta^1 \circ \OLSoptheta^2$, proving part \ref{thm_OLS4}. \hfill \qed

				\subsection{Proof of Lemma \ref{lem_PPcty}} \label{lem_PPctyproof}
					The residual measure $\hat \P_{\hat \Mech, \theta}^\perp$ is supported on the space $\bar{\hat U_{\Mech, \theta}^\perp} \subseteq V$, which is in the kernel of $\Mech : U_\theta \to U_{\Mech, \theta}$ by the Main Theorem. Consequently, the translated measure $\hat \P_{\hat \Mech, \theta}^\perp$ is supported on the space $\hat \Mech(y) + \bar{\hat U_{\Mech, \theta}^\perp}$. Since $\hat \Mech$ is an estimator, $\Mech\big( \hat \Mech(y) \big) = y$, hence $\PP_{\hat \Mech, \theta|\Mech=y}$ is supported on the fiber $\Mech^{-1}(y)$.
					
				Let $\EE_{\hat \Mech,\theta|\Mech=y}$ be the expectation operator of $\PP_{\hat \Mech,\theta|\Mech=y}$. Let $s : \Para \to \S$ be a continuous, bounded function, and note that
						\begin{equation} \label{eqn_proofunif1}
							\EE_{\hat \Mech,\theta|\Mech=y}[s] := \int_\Para s(\para) \, \PP_{\hat \Mech,\theta|\Mech=y}(\d \para) := \int_\Para s\big( \hat \Mech(y) + \resid \big) \, \hat \P_{\hat\Mech,\theta}^\perp(\d \resid). \end{equation}
					
					Let $y^\gamma \to y$ be a convergent net in $\supp \P_{\Mech,\theta}$. By the bounded convergence theorem \cite{williams1991probability},
						\begin{equation}
							\int_\Para s(\para) \, \PP_{\hat \Mech,\theta|\Mech=y^\gamma}(\d \para) = \int_\Para s\big( \hat \Mech(y^\gamma) + \resid\big) \, \hat \P_{\hat\Mech,\theta}^\perp(\d \resid) \to \int_\Para s\big( \hat\Mech(y) + \resid \big) \, \hat \P_{\hat \Mech,\theta}^\perp(\d \resid) = \int_\Para s(\para) \, \PP_{\hat\Mech,\theta|\Mech=y}(\d \para), \end{equation}
					since $s$ and $\hat \Mech$ are continuous functions and $s$ is bounded. This proves the claim of continuity.  \hfill \qed


	\subsection{Proof of Proposition \ref{pro_samemeancov}} \label{pro_samemeancovproof}
		
		Let $y \mapsto \P_{\theta|\Mech=y}$ be a measurable disintegration of $\P_\theta$ with respect to $f$.\footnote{By the Disintegration Theorem \cite{leao2004regular}, measurable disintegrations are guaranteed to exist, though they may not be continuous.} Let $y \mapsto \E_{\theta|\Mech=y}$ denote its regular conditional expectation operator. For any functional $\varphi \in X^*$, the disintegration equations for $\P_{\Mech,\OLS,\theta}^*$ and $\P_\theta$ imply that
					\begin{equation} \label{eqn_condmean}
						\int_Y \EE_{\theta|\Mech=y}[\varphi] \, \P_{\Mech,\OLS,\theta}^*(\d y) = \varphi(m_\theta) = \E_\theta[\varphi] = \int_Y \E_{\theta|\Mech=y}[\varphi] \, \P_{\Mech,\theta}(\d y), \end{equation}
				since $\P_{\Mech,\OLS,\theta}^*$ has mean vector $m_\theta$. Since conditional expectations are unique (up to null sets), this proves that $\EE_{\theta|\Mech=y}[\varphi] = \E_{\theta|\Mech=y}[\varphi]$ for $\P_{\Mech,\theta}$-almost every $y$. Since Hilbert space $H_\theta$ is separable, we may find a countable set $\varphi^i$ of functionals for which $\iota_\theta^* \varphi^i$ is dense in $H_\theta$. Therefore \eqref{eqn_condmean} holds for these countably many functionals (up to a single null set in $Y$), which implies that the Pettis mean of $\PP_{\theta|\Mech=y}$ equals that of $\P_{\theta|\Mech=y}$ (up to a null set in $Y$). 
				
				The proof that the conditional covariance operators agree (up to a null set in $Y$) is similar: replace $\varphi$ in the above argument with $(\bar \psi - \E_\theta[\bar \psi])(\varphi - \E_\theta[\varphi])$.  \hfill \qed

\subsection{Proof of Proposition \ref{pro_PPstrongcty}} \label{pro_PPstrongctyproof}
				
					Let $\|\cdot\|_\Para$ and $\|\cdot\|_\Para$ be continuous norms on $\Para$ and $Y$, respectively. Define the operator norm $M_{\Mech,\theta} := \| \OLSoptheta\|_\op$ of the Gauss-Markov estimator (as in Lemma \ref{lem_cty}), and let $\Delta_{\Mech,\theta,\theta'} := \| \OLSoppara{\theta'} - \OLSoptheta\|_\op$ be the operator norm of the difference function.\footnote{The definitions of the operator norms depend on both $\|\cdot\|_\Para$ and $\|\cdot\|_Y$.} By Lemma \ref{lem_cty}, the real-valued functions $\theta \mapsto M_{\Mech,\theta}$ and $(\theta,\theta') \mapsto \Delta_{\Mech,\theta,\theta'}$ are continuous. 
					
					Choose any Lipschitz-continuous, bounded function $s : \Para \to \S$. Let $\lambda > 0$ denote a Lipschitz constant, so that $\big| s(\para) - s(\para') \big| \le \lambda \|\para - \para'\|_\Para$ for all $\para,\para' \in \Para$. Let $\EE_{\theta|\Mech=y}$ denote the expectation operator of $\PP_{\theta|\Mech=y}$.
					
					Let $(\theta^\gamma,y^\gamma) \to (\theta,y)$ be a convergent net in the extended parameter space. It suffices to show that $\EE_{\theta^\gamma|\Mech=y^\gamma}[s] \to \EE_{\theta|\Mech=y}[s]$.\footnote{The Portmanteau Theorem \cite{klenke2008probability} implies weak continuity of the function $(\theta,y) \mapsto \PP_{\theta|\Mech=y}$.} We apply the representation \eqref{eqn_proofunif1} for both $(\theta^\gamma,y^\gamma)$ and $(\theta,y)$, along with the triangle inequality:
						\begin{eqnarray}
							\Big| \EE_{\theta^\gamma|\Mech=y^\gamma}[s] - \EE_{\theta|\Mech=y}[s] \Big| &=& \left| \int_\Para s(\para) \, \PP_{\theta^\gamma|\Mech=y^\gamma}(\d \para) - \int_\Para s(\para) \, \PP_{\theta|\Mech=y}(\d \para) \right| \nonumber \\\
							&=& \left| \int_\Para s\Big( \OLSoppara{\theta^\gamma}(y^\gamma) + \resid \Big) \, \P_{\Mech,\theta^\gamma}^\perp(\d \resid) - \int_\Para s\Big( \OLSoptheta(y) + \resid \Big) \, \P_{\Mech,\theta}^\perp(\d \resid) \right| \nonumber \\
							&\le&  \int_\Para \Big| s\Big( \OLSoppara{\theta^\gamma}(y^\gamma)  + \resid \Big) - s\Big( \OLSoptheta(y) + \resid \Big) \Big| \, \hat \P_{\Mech,\theta^\gamma}^\perp(\d \resid) \label{qua_A} \\
							&& \quad +\quad \left| \int_\Para s\Big( \OLSoptheta(y) + \resid \Big) \, \P_{\Mech,\theta^\gamma}^\perp(\d \resid) - \int_\Para s\Big( \OLSoptheta(y) + \resid \Big) \, \P_{\Mech,\theta}^\perp(\d \resid) \right|. \nonumber
						\end{eqnarray}					
					The on the second line of \eqref{qua_A} vanishes as $\gamma \to \oo$, since $\resid \mapsto s\big( \OLSoptheta(y) + \resid \big)$ is continuous and bounded and the measures $\P_{\Mech,\theta^\gamma}^\perp$ converge weakly to $\P_{\Mech,\theta}^\perp$. To control the quantity on the first line of \eqref{qua_A}, we use our metric assumptions, the fact that $s$ is Lipschitz, and the triangle inequality to calculate:
						\begin{eqnarray}
							\Big| s\Big( \OLSoppara{\theta^\gamma}(y^\gamma) + \resid \Big) - s\Big( \OLSoptheta(y) + \resid \Big) \Big| &\le& \lambda \Big\| \Big( \OLSoppara{\theta^\gamma}(y^\gamma) + \resid \Big) - \Big( \OLSoptheta(y) + \resid \Big) \Big\|_\Para \nonumber \\
							&\le& \lambda \Big\| \OLSoppara{\theta^\gamma}(y^\gamma - y) \Big\|_\Para + \lambda \Big\|  \OLSoppara{\theta^\gamma}(y)  - \OLSoptheta(y) \Big\|_\Para \nonumber \\
							&\le& \lambda \big\|\OLSoppara{\theta^\gamma} \big\|_\op \big\|y^\gamma - y\big\|_Y + \lambda \big\|\OLSoppara{\theta^\gamma} - \OLSoptheta\big \|_\op \big\|y\big\|_Y \nonumber \\
							&=& \lambda M_{\Mech,\theta^\gamma} \big\|y^\gamma - y\big\|_Y + \lambda \Delta_{\Mech,\theta,\theta^\gamma} \big\|y\big\|_Y. \label{lastquant}
						\end{eqnarray}
					Quantity \eqref{lastquant} vanishes as $(\theta^\gamma,y^\gamma) \to (\theta,y)$, since $M_{\Mech,\theta^\gamma} \to M_{\Mech,\theta}$, $\big\|y^\gamma - y\big\|_Y \to 0$, and $\Delta_{\Mech,\theta,\theta^\gamma} \to \Delta_{\Mech,\theta,\theta} = 0$. This completes the proof. \hfill \qed
					
					



	\subsection{Proof of Theorem \ref{thm_UII}} \label{thm_UIIproof}
			

				
				Suppose that $\P_\theta = \P_{\Mech,\OLS,\theta}^*$, and let $f$ be $\Mech$-uncorrelated. We claim that $f$ is $\Mech$-independent. Let $A \in \B(\S)$ and $B \in \B(Y)$ be Borel sets. It suffices to prove that $f^{-1} A$ and $\Mech^{-1} B$ are independent sets (with respect to $\P_\theta$). We calculate
					\begin{eqnarray}
						\P_\theta\!\big( f^{-1} A \cap \Mech^{-1} B \big) &=& \int_\Para 1_{f^{-1} A}\big(\para\big) 1_{\Mech^{-1} B}\big(\para\big) \, \P_\theta(\d \para) = \int_\Para 1_A\big(f \para\big) 1_B\big(\Mech \para\big) \, \P_\theta(\d \para) \nonumber \\
						&=& \int_\Para \int_\Para 1_A\big(f (\lift+\resid)\big) 1_B\big(\Mech (\lift+\resid)\big)  \, \P_{\Mech,\OLS,\theta}^\perp(\d \resid) \P_{\Mech,\OLS,\theta}(\d \lift), \label{thm_UIIproof1}
					\end{eqnarray} 
				since $\P_\theta = \P_{\Mech,\OLS,\theta}^* = \P_{\Mech,\OLS,\theta} * \P_{\Mech,\OLS,\theta}^\perp$ by assumption. We next simplify expression \eqref{thm_UIIproof1}. 
				
				Note that $f$ vanishes on the lifted Cameron-Martin space $\hat A_{\Mech,\theta}$, since $f$ is $\Mech$-uncorrelated. Since $\hat A_{\Mech,\theta} \cap \supp \P_{\Mech,\OLS,\theta}$ is dense in $\supp \P_{\Mech,\OLS,\theta}$ and $f$ is continuous, we have that $f(y) = 0$ almost surely. Similarly, note that $\Mech$ vanishes on the orthogonal space $\hat U_{\Mech,\theta}^\perp$, using the Main Theorem (Theorem \ref{thm_iso}). Since $\hat A_{\Mech,\theta}^\perp \cap \supp \P_{\Mech,\OLS,\theta}^\perp$ is dense in $\supp \P_{\Mech,\OLS,\theta}^\perp$ and $\Mech$ is continuous, we have that $\Mech(\resid) = 0$ almost surely.
				
				Combining these two facts, \eqref{thm_UIIproof1} equals
					\begin{equation} \label{thm_UIIproof2}
						\int_\Para \int_\Para 1_A\big(f \resid\big) 1_B\big(\Mech \upsilon\big)  \, \P_{\Mech,\OLS,\theta}^\perp(\d \resid) \P_{\Mech,\OLS,\theta}(\d \lift) = \int_\Para 1_A\big(f \resid\big) \P_{\Mech,\OLS,\theta}^\perp(\d \resid) \cdot \int_\Para  1_B\big(\Mech \lift \big)  \,  \P_{\Mech,\OLS,\theta}(\d \lift), \end{equation}
				using Fubini's theorem \cite{folland1999real}. This was the important step, and the rest is straightforward calculation. Observe that the first term of \eqref{thm_UIIproof2} equals
					\begin{eqnarray}
						\int_\Para \int_\Para 1_A\big(f \resid\big) \, \P_{\Mech,\OLS,\theta}^\perp(\d \resid) \hat \P_{\Mech,\OLS,\theta}(\d \lift ) &=& \int_\Para \int_\Para 1_A\big(f (\lift+\resid)\big) \, \P_{\Mech,\OLS,\theta}^\perp(\d \resid) \hat \P_{\Mech,\OLS,\theta}(\d \lift) \nonumber \\
						&=& \int_\Para 1_A\big(f \para\big) \P_\theta(\d \para) = \P_\theta\!\big(f^{-1} A\big); 
					\end{eqnarray}
				this is justified using the fact that $\P_{\Mech,\OLS,\theta}$ is a probability measure, Fubini's theorem (again), the fact that $f(\lift) = 0$ (almost surely), and the UII property (again). A similar argument shows that the second term of \eqref{thm_UIIproof2} equals $\P_\theta\!\big( \Mech^{-1} B \big)$. Consequently, $\P_\theta\!\big( f^{-1} A \cap \Mech^{-1} B \big) = \P_\theta\!\big( f^{-1} A \big) \P_\theta\!\big( \Mech^{-1} B \big)$, which proves that $f$ is $\Mech$-independent.
				
				Now, suppose that $f \in \Para^*$ is $\Mech$-uncorrelated but not $\Mech$-independent. Therefore, there exist Borel sets $A \in \B(\S)$ and $B \in \B(Y)$ so that $\P_\theta(f^{-1} A \cap \Mech^{-1} B) \ne \P_\theta(f^{-1} A) \P_\theta(\Mech^{-1} B)$. By following the previous argument in reverse (which uses only the fact that $f$ is $\Mech$-uncorrelated), we have that $\P_\theta(f^{-1} A) \P_\theta(\Mech^{-1} B) = \big( \P_{\Mech,\OLS,\theta} * \P_{\Mech,\OLS,\theta}^\perp \big)\big( f^{-1} A \cap \Mech^{-1} B \big)$. Therefore, $\P_\theta \ne \P_{\Mech,\OLS,\theta}^*$, hence $\P_\theta$ is not UII (with respect to $\Mech$). \hfill \qed

		\subsection{Proof of Theorem \ref{thm_covrep}} \label{thm_covrepproof}


				For each pair $(e,i) \in V^* \times I$, define the kernel array $k_{ei} := k\big( e\dash\delta_i \big) = (k \circ \delta)(e,i) \in U \subseteq V^I$. Let $V^* \otimes_\cov V^*$ denote the tensor product defined by formal combinations of $e \otimes e'$, and equipped with the minimal topology so that the map $e \otimes e' \mapsto k_{ei} \otimes_\Hilb k_{e'i'}$ is continuous for all $(i,i')$. Let $V \otimes_\cov V$ denote the tensor product defined by formal combinations of $v \otimes v'$, and equipped with the minimal topology so that the map $v \otimes v' \mapsto (e \otimes e')[v \otimes v']$ is continuous for all $(e,e') \in V^* \otimes_\cov V^*$.
				
				Since $V^* \otimes_\cov V^*$ is dense in the dual space $(V \otimes_\cov V)^*$,  a tensor is characterized by the action of all tensor-functionals on it. Therefore, for each $(i,i')$, there exists a tensor $c(i,i') \in V \otimes_\cov V$ so that $(e \otimes e')[c(i,i')] = \cov[k_{ei} | k_{e'i'}] = c(e,i; e',i')$. Continuity of $(i,i') \mapsto V \otimes_\cov V$ is assured by joint continuity of $(i,e; i',e') \mapsto k_{ei} \otimes_\Hilb k_{e'i'}$. Symmetry and non-negative-definiteness are assured since $\cov$ is an inner product. \hfill \qed 
	
	\subsection{Proof of Fundamental Separation Lemma (Lemma \ref{lem_funsep})} \label{lem_funsepproof}
			
				Let $(u_0, u_1)$ and $(u_0', u_1')$ be two solution pairs, and let $\w := u_1 - u_0$ and $\w' := u_1' - u_0'$ denote the corresponding separation vectors. We claim that $\w = \w'$, and that $\Sigma$ is convex.
				
				Define the interpolations $u_0(t) := u_0 + t\big( u_0' - u_0 \big)$ and $u_1(t) := u_1 + t\big( u_1' - u_1 \big)$ for $t \in [0,1]$.\footnote{Note that $u_\ell(0) = u_\ell$ and $u_\ell(1) = u_\ell'$ for $\ell = 0,1$.} Clearly, $\big( u_0(t), u_1(t) \big)$ is a solution if and only if the squared-distance function $\eta(t) := \big\| u_1(t) - u_0(t) \big\|^2$ equals $\rho^2$.\footnote{By assumption, $\eta(0) = \|\w\|^2 = \rho^2 = \|\w'\|^2 = \eta(1)$.} 
				
				We calculate 
					\begin{equation} \label{eqn_uid}
						u_1(t) - u_0(t) = u_1 + t\big( u_1' - u_1 \big) - u_0 - t\big( u_0' - u_0 \big) = u_1 - u_0 + t \big( (u_1' - u_0') - (u_1 - u_0) \big) = \w + t \big( \w' - \w \big). 	\end{equation}
				Applying this identity, we calculate that
					\begin{eqnarray}
						\eta(t) = \big\| u_1(t) - u_0(t) \big\|^2 &=& \big\langle \w + t \big( \w' - \w \big), \w + t \big( \w' - \w \big) \big\rangle \nonumber \\ &=& \|\w\|^2 + 2t \Re\!\big\langle \w' - \w, \w \big\rangle + t^2 \big\| \w' - \w \big\|^2. \end{eqnarray}
				Note that $\eta(0) = \rho^2 = \eta(1)$, and $\big\| \w' - \w \big\|^2 \ge 0$. Since $\eta(t)$ is a quadratic polynomial in $t$, if the leading coefficient $\big\| \w' - \w \big\|^2$ is positive, then for all $t \in (0,1)$, $\big\| u_1(t) - u_0(t) \big\|^2 < \delta^2$. This is a contradiction to the fact that the separation distance is exactly equal to $\delta$. Therefore, $\big\| \w' - \w \big\|^2 = 0$, hence $\w' = \w$, proving \ref{lem_funsep2}. It follows that $\eta(t) = \|\w\|^2 = \rho^2$ for all $t$, so $\big( u_0(t), u_1(t) \big)$ is a solution. This demonstrates convexity, proving part \ref{lem_funsep1}.
				
				Since $\rho > 0$, the compact convex sets $\hat D_0$ and $\hat D_1$ are disjoint, hence separable. Let $(u_0, u_1)$ be a solution pair, and let $M_\ell$ be the affine hyperplane which meets $u_\ell$ and is orthogonal to $\w$. It is clear that $M_0$ and $M_1$ separate $\hat D_0$ and $\hat D_1$ (as in footnote \ref{foot_sep}), otherwise there would exist a solution separated by distance $< \rho$, a contradiction. Each hyperplane $M_\ell$ must contain the solution facet $\Sigma_\ell$, otherwise there would be a solution separated by distance $> \rho$. This proves part \ref{lem_funsep3}.
				
				If the support-vector set $x_\ell := M_\ell \cap (k \circ \delta)D_\ell$ were empty, we would not be able to write a solution $u_\ell \in \Sigma_\ell$ in terms of any data vectors from $(k \circ \delta)D_\ell$. This contracts convexity of $\hat D_\ell$, hence $x_\ell \ne \varnothing$. Since $M_\ell$ is closed and $(k \circ \delta)D_\ell$ is compact, the support-vector $x_\ell$ is compact. Since the support vectors are the extreme points of the convex set $M_\ell \cap \hat D_\ell$, each solution $u_\ell$ can be written as a convex combination of support vectors from $x_\ell$. Consequently, for each $u_\ell$, there exists a (unique) probability measure $\p_\ell$ supported on $x_\ell$ so that $u_\ell = \int_{x_\ell} x_\ell \, \p_\ell(x_\ell)$.\footnote{This is a Pettis integral, so $\langle v, u_\ell = \int_{x_\ell} \langle v, x_\ell \rangle \, \p_\ell(x_\ell)$ for all $v \in U$.} This proves part \ref{lem_funsep4}.
				
				The representations \eqref{eqn_wdec1} and \eqref{eqn_wdec2} follow immediately. If $(u_0 + h, u_1 + h)$ is another solution pair, then the $\nu_\ell$-measures are shifted by $h$. Since $w = (u_1 + h) - (u_0 + h) = u_1 - u_0$, this adjustment is canceled, so the representations are unique. This proves part \ref{lem_funsep5}. \hfill \qed

\bibliographystyle{plain}
\bibliography{biblio_TLM}

\end{document}